\documentclass[11pt,a4paper,reqno,dvipsnames]{amsart}
\usepackage{geometry}
\usepackage{amsfonts}
\usepackage{amsmath}
\usepackage{amssymb,stmaryrd}
\usepackage{amsthm}
\usepackage{amsxtra}
\usepackage{bbm} 
\usepackage{braket}
\usepackage{comment}
\usepackage{extarrows}
\usepackage{graphicx}
\usepackage{latexsym}
\usepackage{mathrsfs}
\usepackage{yhmath}
\usepackage{nicematrix}
\usepackage{mathtools}
\let\underbrace\LaTeXunderbrace

\usepackage{float}
\usepackage{xcolor}
\usepackage{verbatim}
\usepackage{listings}
\usepackage{relsize}

\usepackage[normalem]{ulem}

\usepackage{tikz}
\usetikzlibrary{patterns}
\usetikzlibrary{decorations.pathreplacing, decorations.markings}
\tikzset{->-/.style={decoration={markings,mark=at position #1 with {\arrow{>}}},postaction={decorate}}}

\usetikzlibrary{arrows}
\definecolor{qqwuqq}{rgb}{0,0.39215686274509803,0}
\definecolor{qqqqff}{rgb}{0,0,1}
\definecolor{uuuuuu}{rgb}{0.26666666666666666,0.26666666666666666,0.26666666666666666}
\definecolor{xdxdff}{rgb}{0.49019607843137253,0.49019607843137253,1}
\definecolor{ududff}{rgb}{0.30196078431372547,0.30196078431372547,1}
\definecolor{zzttqq}{rgb}{0.8,0,0.4}
\definecolor{qqwuqq}{rgb}{0,0.39215686274509803,0}
\definecolor{qqqqff}{rgb}{0,0,1}
\definecolor{uuuuuu}{rgb}{0.26666666666666666,0.26666666666666666,0.26666666666666666}
\definecolor{xdxdff}{rgb}{0.49019607843137253,0.49019607843137253,1}
\definecolor{ududff}{rgb}{0.30196078431372547,0.30196078431372547,1}
\usepackage{pgfplots}
\pgfplotsset{compat=1.15}
\usepackage{ytableau}

\usepackage{amsfonts}
\usepackage{amsmath}
\usepackage{amssymb,stmaryrd}
\usepackage{amsthm}
\usepackage{amsxtra}
\usepackage{bbm}
\usepackage{braket}
\usepackage{comment}
\usepackage{extarrows}
\usepackage{graphicx}
\usepackage{latexsym}
\usepackage{mathrsfs}
\usepackage{yhmath}
\usepackage{nicematrix}
\usepackage{mathtools}
\let\underbrace\LaTeXunderbrace

\usepackage{float}
\usepackage{booktabs}
\usepackage{xcolor}
\usepackage{verbatim}
\usepackage{listings}
\usepackage{longtable}
\usepackage[normalem]{ulem}

\usepackage[pdfborder={0 0 0}]{hyperref} 
\hypersetup{colorlinks, citecolor=blue, linkcolor=blue, urlcolor=blue}

\usepackage{bm}

\usepackage{tikz}

\usepackage[initials,lite]{amsrefs} 
\AtBeginDocument{%
    \def\MR#1{}
}

\usepackage{cleveref}

\makeatother

\theoremstyle{plain}

\newtheorem{Theorem}{Theorem}[section]

\newtheorem{Lemma}[Theorem]{Lemma}
\newtheorem{Corollary}[Theorem]{Corollary}
\newtheorem{Proposition}[Theorem]{Proposition}

\theoremstyle{definition}
\newtheorem{Discussion}[Theorem]{Discussion}
\newtheorem{Assumptions and Discussion}[Theorem]{Assumptions and Discussion}

\newtheorem{Example}[Theorem]{Example}
\newtheorem{Definition}[Theorem]{Definition}

\newtheorem{Remark}[Theorem]{Remark}

\newtheorem{Notation}[Theorem]{Notation}
\newtheorem{Construction}[Theorem]{Construction}
\newtheorem{Setting}[Theorem]{Setting}

\theoremstyle{remark}

\newtheorem*{acknowledgment*}{Acknowledgment}

\usepackage[shortlabels]{enumitem}
\SetEnumerateShortLabel{a}{\textup{(\alph*)}}
\SetEnumerateShortLabel{A}{\textup{(\Alph*)}}
\SetEnumerateShortLabel{1}{\textup{(\arabic*)}}
\SetEnumerateShortLabel{i}{\textup{(\roman*)}}
\SetEnumerateShortLabel{I}{\textup{(\Roman*)}}

\def\deg{\operatorname{deg}}

\def\dim{\operatorname{dim}}

\def\Ht{\operatorname{ht}} 

\def\ini{\operatorname{in}} 

\def\ker{\operatorname{ker}}
\def\KK{{\mathbb K}}

\def\MC{\operatorname{MC}}

\def\NN{{\mathbb N}}

\def\part{\operatorname{part}}

\def\ZZ{{\mathbb Z}}

\newcommand\bda{{\bm a}}

\newcommand\bdb{{\bm b}}

\newcommand\bdc{{\bm c}}

\newcommand\bde{{\bm e}}

\newcommand\bdu{\bm{u}}

\newcommand\bdv{\bm{v}}

\newcommand\bdX{{\bm X}}
\newcommand\bdx{{\bm x}}

\newcommand\bfA{\mathbf{A}}

\newcommand\bfB{\mathbf{B}}

\newcommand\bfC{\mathbf{C}}

\newcommand\bfT{\mathbf{T}}

\newcommand\bfX{\mathbf{X}}

\newcommand\bfS{\mathbb{S}}

\newcommand\bflam{\boldsymbol{\lambda}}
\newcommand\bfmu{\boldsymbol{\mu}}

\newcommand\calA{\mathcal{A}}
\newcommand\calB{\mathcal{B}}

\newcommand\calE{\mathcal{E}}
\newcommand\calF{\mathcal{F}}
\newcommand\calG{\mathcal{G}}

\newcommand\calJ{\mathcal{J}}
\newcommand\calK{\mathcal{K}}
\newcommand\calL{\mathcal{L}}

\newcommand\calP{\mathcal{P}}

\newcommand\calR{\mathcal{R}}

\newcommand\calV{\mathcal{V}}

\newcommand{\pd}{\operatorname{pd}}

\newcommand{\rank}{\operatorname{rank}}
\def\reg{\operatorname{reg}}

\newcommand{\si}[1]{{\rm{si}}{\ensuremath{(#1)}}}

\definecolor{MyGreen}{RGB}{34,136,51}

\begin{document}

\title[Algebraic invariants of ladder determinantal modules]{Algebraic invariants of special fiber rings of ladder determinantal modules}

\author[A. Costantini, L. Fouli, K. Goel, K.-N. Lin, H. Lindo, W. Liske, M. Mostafazadehfard]{Alessandra Costantini, Louiza Fouli,  Kriti Goel, Kuei-Nuan Lin, Haydee Lindo, Whitney Liske, Maral Mostafazadehfard}

\thanks{2020 {\em Mathematics Subject Classification}.
    Primary 13A30, 
    13C40, 
    13F65; 
    Secondary 
    14M12, 
    13F50, 
    05E40.  
    }

\thanks{Keywords: Special fiber ring, ladder determinantal modules, regularity, $a$-invariant, multiplicity, analytic spread, linear quotients, distributive lattice, maximal cliques, standard skew Young tableaux}

\address{Department of Mathematics, Tulane University, New Orleans, LA 70118, USA}
\email{acostantini@tulane.edu}

\address{Department of Mathematical Sciences, New Mexico State University, Las Cruces, NM 88003, USA}
\email{lfouli@nmsu.edu}

\address{Department of Mathematics, Indian Institute of Technology Jodhpur, Karwar, Rajasthan, 342030, India}
\email{kritigoel.maths@gmail.com}

\address{Department of Mathematics, The Penn State University, McKeesport, PA, 15132, USA}
\email{linkn@psu.edu}

\address{Department of Mathematics, Harvey Mudd College, Claremont, CA, 91711, USA}
\email{hlindo@hmc.edu}

\address{Department of Mathematics, Saint Vincent College, Latrobe, PA, 15650, USA}
\email{whitney.liske@stvincent.edu}

\address{Institute of Mathematics, Federal University of Rio de Janeiro, RJ,
21941-909, Brazil}
\email{maral@im.ufrj.br}

\begin{abstract}

We provide explicit formulas for key invariants of special fiber rings of ladder determinantal modules, that is, modules that are direct sums of ideals of maximal minors of a ladder matrix. Our results are given in terms of the combinatorial data of the associated ladder matrix. In particular, we compute its dimension, regularity, $a$-invariant, and multiplicity, which via \textsc{Sagbi}  degeneration coincide with those of  Hibi rings associated to a distributive lattice. Then, via  Gr\"{o}bner degeneration these calculations are reduced to those of quotients of polynomial rings by monomial ideals.  Our formula for the multiplicity of the special fiber ring of these ladder determinantal modules is obtained by counting the number of standard skew Young tableaux associated to a certain skew partition, and so provides a natural generalization of the classical formula for the degree of the Grassmannian. 
\end{abstract}

\maketitle

\section{Introduction}

In this paper, we determine the Castelnuovo–Mumford regularity, dimension, $a$-invariant, and multiplicity of the special fiber rings of \emph{ladder determinantal modules}, that is, modules that are direct sums of ideals of maximal minors of ladder  matrices. A ladder matrix is a matrix whose entries are either zero or distinct indeterminates and whose nonzero entries appear in the shape of a ladder, see \cite{ConcaLadder}.  Ladder matrices are an important subclass of sparse matrices in the sense of Giusti-Merle \cite{Giusti-Merle}.   The ideals of maximal minors of such matrices have been extensively studied since the 1980s, with foundational results by De Concini, Eisenbud, and Procesi \cite{DEP}, and subsequent expansions by Conca \cite{ConcaLadder} and others. For arbitrary modules, computing the aforementioned algebraic invariants is difficult, but these calculations can become tractable in cases when the ideal or the module has particularly neat combinatorial features. 
For ladder determinantal modules, we use the combinatorial structure of the ladder matrix to give formulas for the invariants above in terms of the data of the ladder matrix. 

Rees algebras and special fiber rings play a central role in understanding the interplay between algebra and geometry.  The Rees algebra $\calR(I)\coloneqq R[It]$ of an ideal $I$ in a standard graded polynomial ring $R$ over a field $\KK$ is a bigraded subalgebra of $R[t]$, where $t$ is a new indeterminate. In turn, the special fiber ring of $I$ is the $\KK$-algebra $\calF(I)\coloneqq\calR(I)\otimes_{R}\KK$.  Both of these algebras appear in the context of blowup constructions associated with the study of resolution of singularities.  More generally, one can define Rees algebras and special fiber rings of modules, see \cite{EHU, SUVModule}. In particular, Rees algebras and special fiber rings of modules that are direct sums of ideals naturally arise when performing a sequence of successive blowups along disjoint subvarieties. These algebras  can be represented  as quotients of polynomial rings and are often studied through their defining equations, namely the generators of the ideals which define the quotient rings. Then, the algebraic invariants of these algebras can be derived from the defining ideals. 
For instance, the Castelnuovo–Mumford regularity of a finitely generated graded algebra over a standard graded polynomial ring measures the algebra's homological and computational complexity by, for example, providing an upper bound for the largest degree of a minimal generator. Moreover, the multiplicity encodes the asymptotic behavior of Hilbert functions and is a key tool in intersection theory, equisingularity theory, and the study of blowup algebras and integral closures. In algebraic geometry, the multiplicity appears naturally in the context of Bezout's theorem and Hilbert–Samuel functions, offering geometric insight into how varieties intersect and how local properties influence global behavior.

Inspired by Conca, Herzog, and Valla's use of \textsc{Sagbi} degeneration in \cite{CHV96}, the authors of \cite{CDFGLPS} determined the defining equations of the Rees algebra and of the special fiber ring of the ideals of maximal minors of any $2\times n$ sparse matrix; see \cite[Theorem~4.5]{CDFGLPS}. Further, in \cite[Corollaries~4.8,~4.9, 4.10]{CDFGLPS} they identified the dimension, regularity, $a$-invariant, and multiplicity of both the Rees algebra and the special fiber ring by reducing the calculation of these invariants to those of their initial algebras.
A key fact in their proof is that both initial algebras are determinantal rings; consequently, formulas for their invariants follow from well--known results.

The problem of determining the defining equations of blowup algebras of modules that are direct sums of different ideals is broadly open. In fact, to the best of our knowledge, the defining equations for such blowup algebras are only known in a few cases: for direct sums of ideals of linear type, \cite{Ribbe}; for direct sums of different powers of the maximal ideal, \cite{LinPolini}; for direct sums of principal Borel ideals, \cite{dipasquale2020koszul}; for direct sums of two strongly stable ideals with two Borel generators, \cite{KLS}. In particular, the defining equations remain unknown for blowup algebras of direct sums of ideals of minors, even in the simplest case of different ideals of maximal minors of different generic matrices. 

In the case when $M$ is the direct sum of copies of the same ideal of maximal minors of a ladder matrix, Lin and Shen in \cite{LinShenLadder} used \textsc{Sagbi} degeneration to determine the defining equations of the blowup algebras of $M$ and showed that the initial algebra of the special fiber ring of $M$ (with respect to an appropriate monomial order) is a Hibi ring --- a toric algebra associated to a distributive lattice, \cite{HibiDistLatt}. 
While some formulas are known for the dimension, regularity, and multiplicity of Hibi algebras, the existing methods for calculating these invariants require the construction of a distributive lattice and its poset of join-irreducible elements, whose structures can become computationally prohibitive as their size increases. 

Our goal is then to identify alternative methods and give closed formulas which calculate the aforementioned invariants in terms of combinatorial data from the ladder matrix. To obtain such formulas, we use Gr\"obner degeneration, as in the work of Conca and Varbaro \cite{CVGBDeform}, to pass to the initial ideal of the defining ideal of the above Hibi ring. Through various combinatorial techniques, we then obtain formulas for the invariants of this ring. For instance, we define a graph whose maximal cliques allow us to determine its regularity; moreover, we use standard skew Young tableaux to derive a formula for its multiplicity.   Finally, using \textsc{Sagbi} degeneration, we calculate the same invariants for $\mathcal{F}(M)$.

The first step in our work is to identify a combinatorial formula for the analytic spread of our ladder determinantal module $M$.
While the analytic spread of a module is generally difficult to compute, when the module is a direct sum of ideals, the calculation of the analytic spread of the module can be reduced to calculating the analytic spreads of each ideal \cite{BAM}. However, closed formulas for the analytic spread are only known for specific classes of monomial ideals.
Therefore, a combinatorial formula for calculating the analytic spread is highly valuable. Such a formula provides a streamlined path to deriving other key invariants, specifically the reduction number of $M$ and the $a$-invariant of the special fiber ring $\mathcal{F}(M)$, as detailed in \Cref{reg}.

To describe our results, let $\bfX$ be an $n\times m$ ladder matrix, $L$ the ideal of maximal minors of $\bfX$, and $M=\bigoplus_{i=1}^rL$ with $r\in \NN$. In \Cref{analyticSpread} we show that the analytic spread of $M$, i.e. the dimension of $\calF(M)$, is $|\bfX|-n+r$, where $|\bfX|$ denotes the number of variables in $\bfX$. Using the combinatorial structure of the ladder matrix, we associate to the special fiber ring of $M$ a graph $\calG$.
In \Cref{SAGBIandGB} we show that to calculate the relevant invariants for $\calF(M)$ it suffices to compute these invariants for a quotient ring, $T/I(\calG^{\complement})$, where $I(\calG^{\complement})$ is the edge ideal of the complement graph of $\calG$ in a polynomial ring $T$. To compute the regularity of $T/I(\calG^{\complement})$, we  first show in \Cref{Thm: Linear Quotients} that the Alexander dual of $I(\calG^{\complement})$ has linear quotients, and then obtain a formula for the regularity in terms of the analytic spread and the \emph{sequence index} (see \Cref{si Index}) of a particular maximal clique of $\calG$, see \Cref{reg}. 
We remark that the method of using the linear quotient property to calculate the regularity can also be applied in other settings; for instance, it has been used in the case of ideals of rational normal scrolls \cite{LSCMNormal}, ideals of Veronese type \cite{LSVer}, and direct sums of lexsegment ideals \cite{CosLS}. 

To calculate the multiplicity of $\calF(M)$, we first show that it suffices to calculate the multiplicity of $\calF(L)$. For this, we prove that there is a one-to-one correspondence between the maximal cliques of $\calG$ and the standard skew Young tableaux of a certain skew partition; the number of such tableaux then coincides with the multiplicity of $\calF(L)$, see Theorems~\ref{thm: clique skew tableau}, \ref{thm:multiplicity}. 
Our formula generalizes the known formula for the degree of the Grassmannian, that is the number of standard Young tableaux of a given partition; see for instance \cite[Theorem 2.31]{Grassmannians}.

 Finally, as a consequence of our results, we are able to express the rank and the cardinality of the poset of join-irreducible elements of the distributive lattice associated with the Hibi ring, without explicitly constructing the poset; for instance the cardinality of the poset is $|\bfX|-n+r-1$, see \Cref{cor: regposet}.

\section{Preliminaries and notations}
\label{sec:prelim}

Throughout, let $\KK$ be a field. Although this assumption is not always necessary, for simplicity, we assume that $\KK$ has characteristic zero. If $p$ and $q$ are two positive integers such that $p\le q$, we denote the set $\{p,p+1,\dots,q-1,q\}$ by the interval $[p,q]$. When $p=1$, we denote the interval $[1,q]$ by $[q]$.

Let $R = \KK[x_1,\ldots,x_d]$ be the polynomial ring over $\KK$ in the variables $x_1,\ldots,x_d$ and let $I_1, \dots, I_r$ be equigenerated ideals in $R$, with $r \ge 1$. We examine the blowup algebras associated with the module $M= \bigoplus_{j=1}^r I_j$. 
For this purpose, for each $j\in [r]$, let $\{f_{1,j}, \dots,f_{\mu_j,j}\}$ be a minimal homogeneous generating set of $I_j$.
The \emph{Rees algebra}, $\mathcal{R}(M)$, and the \emph{special fiber ring}, $\mathcal{F}(M)$, of such a module $M$ are defined as follows: 
\begin{align*}
\mathcal{R}(M) &\coloneqq R[I_1 t_1, \ldots,I_r t_r ] = R[f_{i,j}t_j : i \in [\mu_j], j \in [r] ]\\
& \, =\bigoplus_{a_i \in \mathbb{Z}_{\ge 0}} I_1^{a_1}\cdots I_r^{a_r}t_1^{a_1}\cdots t_r^{a_r} 
   \, \subseteq R[t_1,\ldots,t_r], \\ 
   \smallskip
\mathcal{F}(M) & \coloneqq \calR(M) \otimes_R \KK  \cong \KK[f_{i,j}t_j : i \in [\mu_j], j \in [r] ].   
\end{align*}  
The algebras $\mathcal{R}(M)$ and $\mathcal{F}(M)$ are also known as the \emph{multi--Rees algebra} and \emph{multi--fiber ring} of the ideals $I_1, \ldots, I_r$, respectively.
One can realize these two algebras as quotients of polynomial rings. Let $\bfT = \{T_{i,j}:  i \in [\mu_j], j \in [r] \}$ be a set of indeterminates over $R$. Consider the polynomial rings $ R[\bfT]$, $\KK[\bfT]$, and the following surjective homomorphisms:
\begin{align*}
    \phi: R[\bfT] \longrightarrow \mathcal{R}(M), \quad \psi: \KK[\bfT] \longrightarrow \mathcal{F}(M),
\end{align*}
determined by $\phi(T_{i,j})=f_{i,j}t_j=\psi(T_{i,j})$ for each  $i\in [\mu_j]$ and $j\in [r]$.
The ideals $\ker (\phi)$ and $\ker (\psi)$ are called the \emph{defining} or \emph{presentation ideals} of $\mathcal{R}(M)$ and $\mathcal{F}(M)$, respectively. In particular, $\calR(M) \cong R[\bfT]/\ker(\phi)$ and $\calF(M) \cong \KK[\bfT]/\ker(\psi)$.

The dimension of the special fiber ring is an important invariant of the module $M$, called the \emph{analytic spread} of $M$ and denoted as $\ell(M) \coloneqq \dim(\mathcal{F}(M))$. The analytic spread coincides with the minimal number of generators of any minimal reduction $U$ of $M$. Recall, that for a module $M$ a minimal reduction (minimal with respect to inclusion) is a submodule $U$ of $M$ such that $[\mathcal{R}(M)]_{s+1} = U [\mathcal{R}(M)]_s$ for some $s\ge 0$.  Moreover, for a minimal reduction $U$ of $M$, the smallest $s$ such that the above equality holds is called the reduction number of $M$ with respect to $U$, denoted $r_U(M)$. 
This number, for instance, controls the generating degree of $\mathcal{R}(M)$ as a module over $\mathcal{R}(U)$. The \emph{reduction number} of $M$ is $r(M)= \min \{r_U(M): U \mbox{ a minimal reduction of } M\}$, see also \cite[Section 8.2]{HS06} and \cite[Section 8]{Vasconcelos2005}. 

When $\calF(M)$ is Cohen--Macaulay, the analytic spread and reduction number of $M$ are related with two important homological invariants of $\calF(M)$, namely its \emph{Castelnuovo--Mumford regularity}, $\reg(\calF(M))$, and its \emph{$a$-invariant}, $a(\calF(M))$.
Recall that $\reg(\calF(M)) \coloneqq \max \Set{j: \beta_{i,i+j}\ne 0}$,  where $\beta_{i,j}$ denote the graded Betti numbers of the module $\calF(M) \cong T/\ker(\psi)$.   To define the $a$-invariant of $\calF(M)$, one first considers the Hilbert series $H_{\calF(M)}(t)$, which can be written as a rational function $Q_{\calF(M)}(t)/ (1-t)^{\ell(M)}$ with $Q_{\calF(M)}(t)\in \ZZ[t]$. Then the $a$-invariant of $\calF(M)$ is $a(\calF(M))\coloneqq\deg Q_{\calF(M)}(t)-\ell(M)$; see also \cite[Definition 3.1.4]{GWGraded}. Moreover, the \emph{multiplicity} of $\calF(M)$, $e(\calF(M))$, is given by $e(\calF(M)) \coloneqq Q_{\calF(M)}(1)$.

The following proposition relates $\ell(M)$ and $r(M)$ with the $a$-invariant and regularity of $\mathcal{F}(M)$. The result is well--known to experts; see, e.g., \cite[Proposition 6.6]{CNPY} and \cite[Proposition 2.2.9]{JonathanThesis} for a proof in the case when $r=1$. We include here a proof in the case when $r>1$ for lack of a specific reference.

\begin{Proposition} \label{CosLS:2.18}
 Let $r\in \mathbb{N}$. For $i \in [r]$, let $I_i$ be nonzero ideals in $R= \mathbb{K}[x_1, \ldots, x_d]$ and let $M=\bigoplus_{i=1}^{r}I_i$. If $\mathcal{F}(M)$ is Cohen--Macaulay, then
$$\reg( \mathcal{F}(M)) = r(M) \quad \text{and} \quad a(\mathcal{F}(M)) = \reg( \mathcal{F}(M)) - \ell(M) = r(M) - \ell(M).$$
\end{Proposition}

\begin{proof}
We regard $\calF(M)$ as a standard graded ring over $R$.
 Note that since $\mathbb{K}$ is an infinite field, there exists a minimal reduction $U$ of $M$ such that $r_U(M) = r(M)$. Let $g_1, \ldots, g_{\ell}$ be a minimal generating set for $U$, where $\ell=\ell(M)$.  The images of these elements in $[\mathcal{F}(M)]_1$ form a linear system of parameters of $\mathcal{F}(M)$, hence a linear regular sequence as $\mathcal{F}(M)$ is Cohen--Macaulay.
 Therefore, $a(\mathcal{F}(M)) = a(\mathcal{F}(M)/U \mathcal{F}(M)) - \ell$ (see for instance \cite[Corollary 3.6.14]{BHCMbook}). 
 Now, as $\mathcal{F}(M)/U \mathcal{F}(M)$ is an Artinian algebra, it follows that
   $$ a(\mathcal{F}(M)/U \mathcal{F}(M)) = \max \{ k \in \mathbb{Z} : [\mathcal{F}(M)/U \mathcal{F}(M)]_k \neq 0\}$$
(see, e.g., \cite[p.~143]{BHCMbook}). On the other hand, using \cite[Lemma 2]{Nagel1990} and \cite[Lemma 2.1]{Nagel2005} as in the proof of \cite[Proposition 6.6]{CNPY}, we get that
$$ \reg(\mathcal{F}(M)) = \reg(\mathcal{F}(M)/U \mathcal{F}(M)) = \max \{ k \in \mathbb{Z} : [\mathcal{F}(M)/U \mathcal{F}(M)]_k \neq 0\}. $$
Since $U$ is a minimal reduction of $M$, it then follows that $[\mathcal{F}(M)/U \mathcal{F}(M)]_k=0$ if and only if $k > r_U(M)=r(M)$. 
Hence, $a(\mathcal{F}(M)/U \mathcal{F}(M))=\reg(\mathcal{F}(M)) = r(M)$. 
\end{proof}

We now discuss ladder determinantal modules. 

\begin{Definition}
    \label{def:Ladder}
   Let $m,n \in \NN$ with $n\le m$. For each $i\in [n]$ let $u_i, v_i\in \NN$ with $u_i\le v_i$ and for every $j\in [u_i,v_i]$ let $x_{i,j}$ be an indeterminate over the field $\KK$.
   \begin{enumerate}[a]
   \item Let $\bfX$ be the $n\times m$ matrix whose $(i,j)$-th entry is $x_{i,j}$ when $j\in S_i\coloneqq [u_i,v_i]$ and $0$ otherwise. When $1=u_1\le  u_2 \le  \cdots \le  u_n\le m$ and $1\le v_1\le  v_2 \le \cdots \le v_n=m$, then the matrix $\bfX$ has the shape of a ladder. We call such a matrix a \emph{ladder matrix} associated to $\bfS \coloneqq \{S_1, \ldots, S_n\}$ and denote the matrix by $\bfX_{\bfS}$.

   \item For a ladder matrix $\bfX_{\bfS}$ let $R\coloneqq\KK[\bfX_{\bfS}]\coloneqq\KK[x_{i,j}: i\in[n], j\in S_i]$ be a polynomial ring over $\KK$ and $L\coloneqq I_n(\bfX_{\bfS})$ be the $R$-ideal generated by the $n\times n$ (maximal) minors of $\bfX_{\bfS}$. 

\item Let $r\in \NN$. Then the module  $M=\bigoplus_{i=1}^rL$ is called a \emph{ladder determinantal module} associated to the matrix $\bfX_{\bfS}$.
   \end{enumerate}
\end{Definition}   
   
  We note here that the ideals we consider are ideals of maximal minors, which are not necessarily the same as what is often referred to as \emph{ladder determinantal} ideals. Ladder determinantal ideals are instead ideals generated by the $t\times t$ minors that fit within the ladder matrix for some fixed $t$, see for example \cite{ConcaLadder}. The following example illustrates the definition above. We will revisit this example often in the rest of the paper. 

\begin{Example}\label{LadderMatrix}
    The $3\times 9$ matrix
    $$\bfX_{\bfS}=\left[\begin{array}{ccccccccc}
                x_{1,1}  & x_{1,2} &x_{1,3} & x_{1,4} &x_{1,5} & 0 &0 &0 &0\\
                0 & 0  &x_{2,3} &x_{2,4} &x_{2,5} &x_{2,6} & 0 &0 &0\\
                0  &0  &0 &x_{3,4} &x_{3,5} &x_{3,6}&x_{3,7}&x_{3,8} &x_{3,9}
        \end{array}\right]$$
    is a ladder matrix associated
    to $\bfS = \{S_1,  S_2,  S_3\}$,
    where $S_1=[1,5]$, $S_2=[3,6]$, $S_3=[4,9]$ and $L=I_3( \bfX_{\bfS})$ is the ideal of all $3 \times 3$ minors of $\bfX_{\bfS}$. Notice that, instead, the $3 \times 3$ ladder determinantal ideal as in \cite{ConcaLadder} in this case is the zero ideal.
\end{Example}

Next we discuss the special fiber ring of ladder determinantal modules.

\begin{Discussion}
\label{discussionFiber}
Adopt the notation of \Cref{def:Ladder}. Since $L$ is a \emph{sparse} determinantal ideal in the sense of \cite{Giusti-Merle}, it follows from a result of Boocher \cite[Proposition 5.4]{Boocher} that the  nonzero maximal minors of $\bfX_{\bfS}$ form a universal Gr\"{o}bner basis for $L$. With this in mind, in light of \textsc{Sagbi} degeneration theory, in order to study the algebra $\calF(M)$ it is useful to consider the special fiber $\calF(N)$ of the module $N\coloneqq\bigoplus_{i=1}^r\ini_{\tau} (L)$, where $\tau$ is a fixed monomial order on $R= \KK[\bfX_{\bfS}]$. 

 Let  $\calL\coloneqq \{\bdc=(c_1,\ldots, c_{n}) \in S_1\times \cdots \times S_n : c_i<c_{i+1} \mbox{ for all } i\in [n-1]\}$. For any tuple $\bdc=(c_1,\ldots, c_{n}) \in \calL$, let $\det[\bdc]$ be the determinant of the $n\times n$ submatrix of $\bfX_{\bfS}$ whose columns are $c_1, \ldots, c_n$. Notice that 
      by the ladder shape of $\bdX_{\bfS}$ we have  $\det[\bdc]\neq 0$ for any $\bdc\in \calL$. Therefore, $\{ \det[\bdc] : \bdc\in \calL\}$ is a minimal generating set of the ideal $L$. Hence, the natural generators of the module $N$ correspond to the tuples $(\bdc,i)$ in the Cartesian product $\calL\times [r]$. Furthermore, since the set $\{\det[\bdc] : \bdc\in \calL\}$ forms a reduced Gr\"{o}bner basis for $L$, it follows that  
      $$\calF(N) \cong \KK[\ini_{\tau}(\det [\bdc])t_i :  \bdc \in \calL, i\in [r]].$$
 Then,  there is a surjective homomorphism from the polynomial ring $T \coloneqq \KK[T_{\bdc,i} : (\bdc,i)\in \calL\times [r]]$ to the special fiber ring $\calF(N)$, namely
$$ \rho: T \longrightarrow \mathcal{F}(N), $$
given by $\rho(T_{\bdc,i})=\ini_{\tau}(\det[\bdc])t_i$ for all $\bdc\in \calL$ and all $i\in [r]$. Let  $\calJ \coloneqq\ker(\rho)$ be the presentation ideal of $\mathcal{F}(N)$. 
We also consider the map $\psi : T \longrightarrow \mathcal{F}(M)$ given by $\psi(T_{\bdc,i})= (\det[\bdc])t_i$, for all $\bdc \in \calL$ and all $i\in [r]$;
then $\calK \coloneqq \ker (\psi)$ is the presentation ideal of $\calF(M)$.
\end{Discussion}

Following the work of Lin and Shen \cite{LinShenLadder}, one can select monomial orders $\tau$ and $\sigma$ on $R$ and $T$ respectively, so that via \textsc{Sagbi} basis degeneration theory one can relate and identify the defining equations of $\calF(N)$ and $\calF(M)$; see \Cref{FiberLadder}. 
Recall that for a $\KK$-algebra $A'$ and a finitely generated $\KK$-subalgebra $A$, the \emph{initial algebra} of $A$, $\ini_\tau(A)$, is the $\mathbb{K}$-subalgebra of $A'$ generated by the initial monomials $\ini_\tau(a)$, for all $a\in A$, where $\tau$ is a monomial order on $A'$.  A set of elements $\calB\subseteq A$ is called a \emph{\textsc{Sagbi} basis for $A$ with respect to $\tau$}, if $\ini_\tau(A) = \KK[\ini_\tau(\calB)]$. For detailed applications of \textsc{Sagbi} bases to the study of blowup algebras, see \cite{CHV96} and \cite{BCresPowers}.

\begin{Definition}\label{order}
   Adopt the notation of \Cref{def:Ladder}. 
    \begin{enumerate}[a]
        \item Let the entries in $\bfX_{\bfS}$ be ordered so that $x_{i,j}>x_{l,k}$ if and only if $i<l$ or $i=l$ and $j<k$ and let $\tau$ be the lexicographic order on the ring $R$ with respect to this total order.
        \item By abuse of notation, let $\tau$ also denote the lexicographic order on $\calL \times [r]$.  That is, for vectors $\bda,\bdb \in \calL \times [r]$, we say that $\bda>_{\tau}\bdb$ if the first nonzero entry of $\bda-\bdb$ is negative.  
    \end{enumerate}    
\end{Definition}

\begin{Discussion} \label{dis: join-meet}
  Adopt the notation of \Cref{def:Ladder}. There is a natural partial order on $\calL \times [r]$, namely, for $\bda= (a_1, \ldots, a_{n+1}), \bdb = (b_1, \ldots, b_{n+1}) \in \calL \times [r]$, set $\bda < \bdb$ if $a_i\le b_i$ for all $i$. With this order, any two elements $\bda, \bdb\in \calL \times [r]$ admit a join and a meet in $\calL \times [r]$. Indeed, $\bda \vee \bdb = \max\{\bda,\bdb\} = (c_1, \ldots, c_{n+1})$ and $\bda \wedge \bdb = \min\{\bda,\bdb\} = (d_1, \ldots, d_{n+1})$, where for each $i \in [n+1]$, $c_i\coloneqq\max\{a_i,b_i\}$ and $d_i\coloneqq\min \{a_i,b_i\}$. Notice that the ladder structure of $\bfX_{\bfS}$ implies that $\bda \vee \bdb, \bda \wedge \bdb \in \calL \times [r]$; see also \cite[Lemma 4.4]{LinShenLadder}. Moreover, with this partial order, $\calL \times [r]$ is a distributive lattice. 

The \emph{join-meet ideal} of $\calL \times [r]$ is the ideal $I_{\calL \times [r]} \coloneqq (T_{\bda}T_{\bdb} - T_{\bda \wedge \bdb} T_{\bda \vee \bdb} \mid \bda, \bdb \in \calL \times [r])$ in $T=\KK[T_{\bda}: \bda\in \calL \times [r]]$. 
Note that $T_{\bda}T_{\bdb} - T_{\bda \wedge \bdb} T_{\bda \vee \bdb} \neq 0$ if and only if $\bda, \bdb$ are incomparable, i.e., if and only if $\bda, \bdb \neq \min \{\bda, \bdb\}$. 
The quotient ring $T/ I_{\calL \times [r]}$ is called the \emph{Hibi ring} of $\calL \times [r]$, see \cite{HibiDistLatt}. We refer the reader to \cite[Chapter 6]{HerHiOh} and \cite[Chapter 6]{EHgbBook} for more information on distributive lattices and Hibi rings.

We say that a monomial order $\sigma$ on $T$ is \emph{compatible} if for any incomparable $\bda, \bdb\in \calL\times [r]$, one has $\ini_{\sigma}(T_{\bda}T_{\bdb} - T_{\bda \wedge \bdb} T_{\bda \vee \bdb})=T_{\bda}T_{\bdb}$. An example of such a compatible monomial order can be constructed as follows. 
With $\tau$ as in \Cref{order}, order the variables of $T$ so that $T_{\bda} > T_{\bdb}$ if and only if $\bda >_{\tau} \bdb$, and choose $\sigma$ to be the graded reverse lexicographic order on $T$ induced by this order on the indeterminates; see \cite[Example~6.16]{HerHiOh}.

Finally, notice that for any compatible monomial order $\sigma$, we have $\ini_{\sigma}(\calJ)=(T_{\bda}T_{\bdb}: \bda, \bdb\neq \min\{\bda, \bdb\})$.  
The monomials $\{T_{\min\{\bda, \bdb\}}T_{\max\{\bda, \bdb\}}: \bda, \bdb\neq \min\{\bda, \bdb\} \}$ are sometimes referred to as the \emph{standard monomials} of $\calJ$ with respect to  $\sigma$, see \cite{Sturmfels}.
\end{Discussion}

Using the notions above, Lin and Shen were able to determine the defining equations of the special fiber ring $\calF(N)$. Furthermore, using the theory of \textsc{Sagbi} bases they identified the defining equations of the special fiber ring $\calF(M)$.

\begin{Theorem}[{\cite[Theorems~4.8, 5.5, Corollaries~4.10, 5.7]{LinShenLadder}}] \label{FiberLadder}
Let $M$ and $N$ be as in \Cref{discussionFiber}.
\begin{enumerate}[a]
    \item The natural generators of $M$ 
    form    a \textsc{Sagbi} basis for $\calF(M)$ with respect to an extended monomial order $\tau'$, that is $\ini_{\tau'}(\calF(M))=\calF(N)$. Moreover, there exists a monomial  order $\omega$ on $T$ such that $\ini_{\omega}( \calK)=\ini_{\sigma} (\calJ)$, where $\calK, \calJ$ are the defining ideals of $\calF(M)$ and $\calF(N)$, respectively.

    \item    The $\KK$-algebra $\calF(N)$ is a Hibi ring and is isomorphic to $T/\calJ$ where 
    $$ \quad \calJ=(T_{\bda}T_{\bdb}-T_{\bda \wedge \bdb}T_{\bda \vee \bdb}: \bda, \bdb \in \calL\times[r]).$$

    \item  $\calF(M)$ and $\calF(N)$ are both Cohen--Macaulay normal domains.
\end{enumerate}
\end{Theorem}

As a consequence of \Cref{FiberLadder}, the study of various algebraic invariants of $\calF(M)$ is reduced to the study of the same invariants for $\calF(N)$, and in fact those of  $T/\ini_{\sigma} (\calJ)$.
Notice that by \Cref{FiberLadder}, $\ini_{\sigma} (\calJ)$ is generated by squarefree monomials of degree 2. We write $(\ini_{\sigma}(\calJ))^{\vee}$ for the Alexander dual ideal of $\ini_{\sigma}(\calJ)$. By combining several well--known results, the following corollary allows us to analyze the algebraic invariants of the ring $T/\ini_{\sigma}(\calJ)$.

\begin{Corollary}\label{SAGBIandGB}
With the setting of \Cref{FiberLadder} we have: 
\begin{enumerate}[a]
    \item \label{analytic} $\ell(M) =\ell(N) =\dim(T/\ini_{\sigma}(\calJ))$.
    \item \label{regToPd} $\reg(\calF(M))= \reg(\calF(N))=\reg(T/\ini_{\sigma}(\calJ))=\pd((\ini_{\sigma}(\calJ))^{\vee})$.
    \item \label{ainvariant} $a(\calF(M))= a(\calF(N))= a(T/\ini_{\sigma}(\calJ))$.
    \item \label{multiplicity eq} $e(\calF(M))=e(\calF(N))=e(T/\ini_{\sigma}(\calJ))$.
\end{enumerate}
\end{Corollary}

\begin{proof}
    By \Cref{FiberLadder}, $\calF(M)$ is Cohen--Macaulay and there exists a monomial order $\omega$ such that $\ini_{\omega}(\calK) = \ini_{\sigma}(\calJ)$; moreover, $\calF(N)$ is Cohen--Macaulay. 
    Thus by \cite[Corollaries 2.7 and 2.11]{CVGBDeform}, we have
    $$ \ell(M) = \dim (\calF(M))=\dim( T/\ini_{\omega}(\calK))=\dim (T/\ini_{\sigma}(\calJ))=\dim(\calF(N)) = \ell(N). $$
    Hence, again by \cite[Corollaries 2.7 and 2.11]{CVGBDeform}, we obtain the following equalities
$$ \reg (\calF(M))= \reg(T/\ini_{\omega}(\calK)) = \reg(T/\ini_{\sigma}(\calJ)) =\reg(\calF(N)). $$
  Since $\ini_{\sigma}(\calJ)$ is a squarefree monomial ideal, by \cite[Proposition 8.1.10]{HHMon} it then follows that $\,\reg (\calF(M))= \reg(T/\ini_{\sigma}(\calJ))=\pd((\ini_{\sigma}(\calJ))^\vee)$.
This proves statements~\ref{analytic} and \ref{regToPd}, which together with \Cref{CosLS:2.18} imply statement~\ref{ainvariant}.  

Now, note that by \cite[Corollary 2.5]{CHV96} and \Cref{FiberLadder}, it follows that $e(\calF(M))=e(\calF(N))$. Moreover, $e(\calF(N))$ can be calculated from the Hilbert series of $T/\calJ$ since $\calF(N)\cong T/\calJ$. Meanwhile, the Hilbert series of $T/\calJ$ and $T/\ini_{\sigma}(\calJ)$ coincide.
Therefore, we have $e(\calF(M))=e(\calF(N))=e(T/\calJ)=e(T/\ini_{\sigma}(\calJ))$. This proves statement~\ref{multiplicity eq}.
\end{proof}

\begin{Remark}
    Notice that, since $\calF(N)$ is a Hibi ring, formulas for $\ell(N)$, $\reg(\calF(N))$ and $a(\calF(N))$ can be given in terms of combinatorial data of the poset of the join-irreducible elements of the lattice $\calL \times [r]$; see \cite[Theorems~6.38 and~6.42]{HerHiOh}. Moreover, by \cite[Theorem 3.9]{BLHibi}, the multiplicity $e(\calF(N))$ is the number of maximal chains in  $\calL \times [r]$, however no explicit formulas are known to calculate this number for arbitrary Hibi rings. 
By shifting the attention on the ring $T/\ini_{\sigma}(\calJ)$ via \Cref{SAGBIandGB}, we will instead express each of these invariants in terms of numerical data of the matrix $\bfX_{\bfS}$ (see \Cref{analyticSpread}, \Cref{reg}, and \Cref{thm:multiplicity}). We will then use our findings to reinterpret the previously known results without needing to investigate the combinatorial structure of the poset of the join-irreducible elements of $\calL \times [r]$; see \Cref{cor: regposet}.
\end{Remark}

\section{The linear quotient property}\label{sec:linear}

The objective of this section is to establish that the Alexander dual ideal $(\ini_{\sigma}(\calJ))^\vee$, with $\calJ$ as in \Cref{FiberLadder}, is equigenerated and has linear quotients. This will allow us to calculate the regularity of $\calF(M)$ in \Cref{sec:reg}. 
Recall that a graded ideal $I$ has \emph{linear quotients} if there exists a system of homogeneous generators $f_1, \ldots, f_{\mu}$ of $I$
 such that for all $i$, the colon ideal $Q_i=\braket{f_1, \ldots, f_{i-1}}:f_i$ is generated by linear forms, see for example \cite[Section~8.2.1]{HHMon}.

We first observe that, in light of \Cref{SAGBIandGB}, our ladder matrix can be simplified as discussed in the following remark. 

\begin{Remark}\label{shrinkmatrix}
Let $\bfX_{\bfS}$ be a ladder matrix associated to $\bfS=\{S_1, \ldots, S_n\}$ as in \Cref{def:Ladder}. For all $i\in [2,n]$, let $S_i'=S_i$ if $u_{i-1}<u_{i}$ and $S_i'=[u_{i}+1, v_i]$ whenever $u_{i-1}=u_i$. Let $\bfS'=\{S_1, S_2' \ldots, S_{n}'\}
$ and let ${\bfX}_{\bfS'}$ be the ladder matrix corresponding to $\bfS'$. Notice $\ini_{\tau}(I_n(\bfX_{\bfS}))=\ini_{\tau}(I_n(\bfX_{\bfS'}))$.
Furthermore, since $\tau$ is a diagonal monomial order we may assume that $u_{i}\le v_{i-1}+1$ for all $i\in[2,n]$. Therefore, without loss of generality, in \Cref{def:Ladder} we may assume that $u_{i-1}<u_{i}$ and similarly $v_{i-1}<v_{i}$ for all $i\in [2,n]$.
\end{Remark}

With this, the following setting will be adopted for the rest of this work.

\begin{Setting} \label{LadderSetting} Let $m,n,r\in \NN$ with $n\le m$ and let $\KK$ be a field of characteristic $0$. 
\begin{enumerate}[a]
\item Let $\bfX_{\bfS}$ be a ladder matrix associated to $\bfS=\{S_1, \ldots, S_n\}$ with $S_i = [u_i, v_i]$, for all $i\in [n]$, and $u_{i-1} < u_{i}$, $v_{i-1}<v_{i}$,  $u_i\le v_{i-1}+1$ for all $i\in [2,n]$.

\item Let $R=\KK[x_{i,j}: i\in [n], j\in S_i]$, where $x_{i,j}$ are indeterminates over $\KK$ and let $L=I_n(\bfX_{\bfS})$ be the $R$-ideal of $n\times n$ minors of $\bfX_{\bfS}$.
\item Let $M=\bigoplus_{i=1}^r L$ and $N=\bigoplus_{i=1}^r \ini_{\tau} (L)$ be the direct sums of $r$ copies of $L$ and $\ini _{\tau}(L)$, respectively, where $\tau$ is the lexicographic order on $R$ as in \Cref{order}.
\item Let  $\calL= \{\bdc=(c_1,\ldots, c_{n}) \in S_1\times \cdots \times S_n : c_i<c_{i+1} \mbox{ for all } i\in [n-1]\}$.
\item  Write $\mathcal{F}(M)\cong T/\calK$ and $\mathcal{F}(N)\cong T/\calJ$, where $T = \KK[T_\bda: \bda\in \calL\times [r]]$ and $\calJ=(T_{\bda}T_{\bdb}-T_{\min\{\bda, \bdb\}}T_{\max\{\bda, \bdb\}}: \bda, \bdb \in \calL\times[r])$; see \Cref{FiberLadder}. 
\item Let $\calG^r$ be the simple graph on the vertex set $\calL \times [r]$ such that  $\{\bda,\bdb\}$ is an edge in $\calG^r$ whenever $\bda=\min\{\bda,\bdb\}$ or $\bdb=\min\{\bda, \bdb\}$. We write $\calG$ when $r$ is clearly understood. 
\item Let $\sigma$ be the graded reverse lexicographic order on $T$ induced by $\tau$ as in \Cref{dis: join-meet}.
   \item  Let $\Delta_i\coloneqq v_i-u_i$ for $i \in [n]$ and $\epsilon_i\coloneqq u_{i+1}-u_i$ for $i \in [n-1]$. Notice that $\Delta_i \geq \epsilon_i-1$ for all $ i\in [n-1]$. Moreover, for all $i\in [2,n]$ we have $\Delta_{i-1}\le \Delta_i+\epsilon_{i-1}-1$.
\end{enumerate}
\end{Setting}

By \Cref{FiberLadder}, $\ini_{\sigma}(\calJ)$ is generated by the monomials $T_{\bda}T_{\bdb}$ such that $\bda, \bdb \neq \min\{\bda, \bdb\}$. Therefore, the squarefree monomial ideal $\ini_{\sigma}(\calJ)$ is the edge ideal of the complement graph ${\calG}^{\complement}$. In order to describe the generators of $(\ini_{\sigma}(\calJ))^{\vee}$, we need to recall the notions of maximal cliques and minimal vertex covers of a graph.

Let $G$ be graph with vertex set $V(G)$ and edge set $E(G)$. A \emph{clique} of $G$ is a collection of vertices such that any two vertices in the collection are adjacent. A \emph{maximal clique} is a clique that is maximal with respect to inclusion. Let $\MC(G)$ be the set of maximal cliques of $G$. A \emph{vertex cover} of $G$ is a subset $C\subseteq V(G)$ such that $\{i,j\}\cap C \neq \emptyset$ for all $\{i,j\} \in E(G)$. A vertex cover $C$ is called \emph{minimal} if no proper subset of $C$ is a vertex cover of $G$.

\begin{Remark}
    \label{rmk:maximal_cliques} 
    Adopt \Cref{LadderSetting}. Since $\ini_{\sigma}(\calJ)$ is the edge ideal of the complement graph $\calG^{\complement}$, then  $T_{\bda^1}T_{\bda^2}\cdots T_{\bda^{p}}$ is a minimal monomial generator of $(\ini_{\sigma}(\calJ))^{\vee}$ if and only if $\{\bda^1, \dots, \bda^p\}$ is a minimal vertex cover of the complement graph $\calG^{\complement}$; see \cite[Corollary 9.1.5]{HHMon}. 
    Equivalently, the set complement $\{\bda^1, \dots, \bda^{p}\}^{\complement}= (\calL \times [r])\setminus \{\bda^1, \dots, \bda^{p}\}$ is a maximal clique of $\calG$. Therefore, the minimal monomial generators of the Alexander dual ideal are in one-to-one correspondence with the maximal cliques of $\calG$.
\end{Remark}

Using the fact that $\calF(M)$ is Cohen--Macaulay, see \Cref{FiberLadder}, we can now show that $(\ini_{\sigma}(\calJ))^{\vee}$ is equigenerated and obtain a formula for the analytic spread of $M$.

\begin{Theorem} \label{analyticSpread}
   Adopt \Cref{LadderSetting}. The analytic spread of $M$ is
   $$\ell(M) = \ell(L) +r-1=r+\sum _{i=1}^{n}\Delta_i.$$
   Moreover, the Alexander dual ideal, $(\ini_{\sigma}(\calJ))^{\vee}$, is equigenerated in degree $\dim T -\ell(M)$, and every maximal clique of $\calG$ has the same cardinality, equal to $\ell(M)$. 
\end{Theorem}

\begin{proof}
 By \Cref{SAGBIandGB}~\ref{analytic}, we have $\ell(M) = \dim (T/\ini_{\sigma}(\calJ))$.
Now, by \cite[Corollaries 2.7 and 2.11]{CVGBDeform} and \Cref{FiberLadder}, we know that $\ini_{\sigma}(\calJ)$ is a Cohen--Macaulay ideal, hence it is height unmixed with $\Ht(\ini_{\sigma}(\calJ)) = \Ht(\calK) = \dim T -\ell(M)$. Therefore, every minimal monomial generator of $(\ini_{\sigma}(\calJ))^\vee$ has degree $\dim T - \ell(M)$, and every maximal clique has size $\dim T - (\dim T - \ell(M)) = \ell(M)$ by \Cref{rmk:maximal_cliques}.  

Thus, it only remains to compute the analytic spread of $M$.  By \cite[Corollary 3.13]{BAM} we know that $\ell(M)=\ell(L)+r-1$. Therefore, it suffices to prove that $\ell(L)=1+\sum _{i=1}^{n}\Delta_i$. 

 By \Cref{SAGBIandGB}~\ref{analytic} it follows that $\ell(L)=\dim(\mathcal{F}(\ini_{\tau}(L)))$. Recall that $\ini_{\tau}(L)$ is generated by $\ini_{\tau}(\det[\bda])$, where $\bda\in \calL$, see \Cref{discussionFiber}. Notice that for each $\bda \in \calL$ we can write $x_{\bda}=x_{1,a_1}x_{2,a_2}\cdots x_{n,a_n}=\ini_{\tau}(\det[\bda])$.
 Let $s=\dim R$ and 
 let $\bde_{i,j}=(0,\ldots,0,1,0,\ldots,0) \in \KK^{s}$ be the standard basis vectors of $\KK^{s}$ corresponding to the exponent vectors of $x_{i,j}$ for $i\in [n], j\in S_i$. Hence for each $\bda=(a_1, \ldots, a_n) \in \calL$ the exponent vector of $x_{\bda}$ is $\sum_{i=1}^n \bde_{i,a_i}$. By \cite[Lemma 4.2]{Sturmfels} and its proof, it follows that $\dim(\calF(\ini_{\tau}(L)))$ is equal to the dimension of the lattice spanned by the exponent vectors of $\bdx_{\bda}$ with $\bda\in \calL$. Moreover, the dimension of this lattice is the maximum number of linearly independent exponent vectors. 

 We consider the following set
       \begin{align*}
     \calB=& \bigcup\limits_{i=2}^{n} \{x_{1,u_1}x_{2,u_2} \cdots x_{i-1,u_{i-1}}x_{i, a_i} x_{i+1,v_{i+1}} \cdots x_{n, v_n} : u_{i}\le a_{i} < v_{i}\} \\
     &\cup \{x_{1,a_1}x_{2,v_2} \cdots x_{n,v_n} : u_1\le a_1 \le v_1\}.
     \end{align*}

      It is clear that in the set $\mathcal{B}$, each element involves a distinct new variable. This makes the set algebraically independent; hence, $\mathcal{B}$ is part of a transcendence basis of the domain $\mathcal{F}(\operatorname{in}_{\tau}(L))$ over the field $\mathbb{K}$,  and thus $\dim \calF(\ini_{\tau}(L))\ge |\mathcal{B}|=1 +\sum_{i=1}^{n} \Delta_i$.

We see that the set of exponent vectors of the elements in $\mathcal{B}$ is 
\begin{align*}
           \calV&=
            \bigcup\limits_{i=2}^{n} \{\sum_{j=1}^{i-1}\bde_{j,u_j} +\bde_{i,a_i} +\sum _{j=i+1}^n \bde_{j,v_j} : u_{i}\le a_{i} < v_{i}\} \cup \{ \bde_{1,a_1} +\sum _{j=2}^n \bde_{j,v_j}: u_1\le a_1 \le v_1\}.
       \end{align*}
Moreover,  the exponent vector of any  $\bdx_{\bda}\in \calL\setminus \mathcal{B}$, where $\bda=(a_1, \ldots, a_n)$, can be written as a linear combination of  vectors in $ \calV$, since

\begin{align*}
    \sum_{i=1}^n \bde_{i,a_i}&=\sum\limits_{i=1}^{n-1}\left[ [\sum_{j=1}^{i-1}\bde_{j,u_j} +\bde_{i,a_i} +\sum _{j=i+1}^n \bde_{j,v_j}]-[ \sum_{j=1}^{i}\bde_{j,u_j}  +\sum _{j=i+1}^n \bde_{j,v_j}] \right]\\
    &+\sum_{i=1}^{n-1}\bde_{i,u_i}  +\bde_{n,a_n}.
\end{align*} 
Therefore, $\ell(L)=|\mathcal{B}|=1+\sum \limits_{i=1}^n\Delta_i$.
\end{proof}

\begin{Example}
\label{ex:analyticspread}
    We consider the matrix in \Cref{LadderMatrix}. We have $\Delta_1=5-1=4, \Delta_2=6-3=3$, and $\Delta_3=9-4=5$. Then, $\ell(L)=1+4+3+5=13$, and if $M=L \bigoplus L$, i.e., $r=2$, then $\ell(M)=14$.
\end{Example}

Given now that all maximal cliques have the same length, we introduce the following notation.

\begin{Notation}
 \label{cliquenotation}
 Adopt \Cref{LadderSetting}. For any clique $\bfA$ in $\calG$, without loss of generality we write $\bfA=[\bda^1,\ldots,\bda^{s}]$ with $\bda^i >_{\tau} \bda^{i+1}$ for all $i\in [s-1]$, where $s\le \ell(M)$. 
\end{Notation} 

We next observe a few immediate properties of the maximal cliques of $\calG$.

\begin{Remark}
     \label{cliqueproperties}
     Adopt \Cref{LadderSetting} and let $\ell=\ell(M)$.
\begin{enumerate}[a]
      \item \label{rem:tuple inequalities} For any $\bda, \bdb\in \calL \times [r]$ with $\{\bda, \bdb\}\in E(\calG)$ we have $\bda>_{\tau}\bdb$ if and only if $\min\{\bda, \bdb\}=\bda$; equivalently,  $\bda_j\le \bdb_j$ for all $j\in [n+1]$. 
Therefore, $\bfA=[\bda^1,\ldots,\bda^{s}]$ is a clique in $\calG$ if and only if $\bda_{j}^{h}\le \bda_{j}^{k}$ for any $j\in [n+1]$ and any $h,k\in [s]$ with $h<k$, or equivalently if $\min\{\bda^h, \bda^k\}=\bda^h$ for any $h,k\in[s]$ with $h<k$.
      \item In particular, \ref{rem:tuple inequalities} implies that the maximal cliques of $\calG$ are the maximal chains in the lattice $\calL \times [r]$.
      \item \label{rem:first and last in every clique} Let $\bdu \coloneqq (u_1,\ldots,u_n,1)$  and  $\bdv \coloneqq (v_1,\ldots,v_n,r)$.  Then, for any $\bdb\in \calL \times [r]$ we have that $\{\bdu, \bdb\}, \{\bdv, \bdb\} \in E(\calG)$, since $\min\{\bdb, \bdu\}=\bdu$ and $\min\{\bdb, \bdv\}=\bdb$.  Thus, for any $\bfA=[\bda^1,\ldots,\bda^{\ell}]\in \MC(\calG)$, it follows that $\bda^1=\bdu$  and $\bda^{\ell}=\bdv$.
  \end{enumerate}
\end{Remark}

The following crucial lemma shows that for any maximal clique $\bfB=[\bdb^1, \ldots, \bdb^{\ell}]$ the difference between any two consecutive tuples is a standard basis vector $\bde_p=(0,\ldots,0,1,0,\ldots,0) \in \KK^{n+1}$, for some $p\in[n+1]$. 

\begin{Lemma}\label{lem:conse}
     Adopt \Cref{LadderSetting}. Let $\bfB=[\bdb^1, \ldots, \bdb^{\ell}]\in \MC(\calG)$, where $\ell=\ell(M)$.    For every $i\in [\ell-1]$ there exists a $p_i\in[n+1]$ such that $\bdb^{i+1}=\bdb^{i}+\bde_{p_i}$.       
\end{Lemma}

\begin{proof}
Let $\bfB=[\bdb^1,\bdb^2, \ldots, \bdb^\ell]$.
Suppose by contradiction that there is an $i \in [\ell-1]$ such that $\bdb^{i+1}-\bdb^{i}\neq \bde_p$ for any $p\in [n+1]$.
Let $\bda\coloneqq\bdb^{i+1}-\bdb^{i}=(a_1, \ldots, a_{n+1})$. By \Cref{cliqueproperties}~\ref{rem:tuple inequalities} we have that $\bdb^{i}>_{\tau}\bdb^{i+1}$ and $\min\{\bdb^i, \bdb^{i+1}\}=\bdb^{i}$.  Therefore, $a_j\ge 0$ for all $j\in [n+1]$ and $\displaystyle \sum_{j=1}^{n+1} a_j>1$ as $\bda \neq \bde_p$ for any $p\in [n+1]$. 

Let $q \coloneqq \min\{j \mid a_j>0\}$, $\bdc \coloneqq \bdb^{i}+\bde_q$, and notice that $\bdc \neq \bdb^{i+1}$. Now, 
$\min\{\bdb^i, \bdc\}=\bdb^i$; moreover,  for all $j \le i$, we have that $\min\{\bdb^{j}, \bdc\}=\bdb^{j}$, since $\min\{\bdb^{j}, \bdb^{i}\}=\bdb^j$. In other words, $\bdb^{i}>_{\tau}\bdc$ and $\{\bdb^{j}, \bdc\}\in E(\calG)$ for all $j\le i$.

 Similarly, by the choice of $q$ we have that $\bdc_k=\bdb^i_k\le \bdb^{i+1}_k$ for all $k\neq q$, $\bdc_q=\bdb^i_q+1 \le\bdb^{i+1}_q$, and $\min\{\bdc, \bdb^{i+1}\}=\bdc$. As above, we conclude that $\bdc>_{\tau}\bdb^{i+1} $ and $\{\bdc, \bdb^{j}\}\in E(\calG)$ for all $j\ge i+1$, since $\min\{\bdb^{i+1}, \bdb^{j}\}=\bdb^{i+1}$ for all $j\ge i+1$.
 
 Now, as $\bdb^{i}>_{\tau} \bdc >_{\tau} \bdb^{i+1}$ and $\{\bdc, \bdb^{j}\}\in E(\calG)$ for all $j\in [\ell]$, then $[\bdb^1, \ldots, \bdb^i, \bdc, \bdb^{i+1}, \ldots, \bdb^{\ell}]$ is a clique of $\calG$, contradicting the maximality of $\bfB$.
\end{proof}

With \Cref{lem:conse} in mind, to each maximal clique $\bfB$ we associate a sequence of positive integers as follows and introduce an order on the set $\MC(\calG)$ of the maximal cliques of $\calG$.

\begin{Definition}\label{positions}
 Adopt \Cref{LadderSetting} and let $\ell=\ell(M)$.
    \begin{enumerate}[a]
      \item Let $\bfB=[\bdb^1, \ldots, \bdb^{\ell}] \in\MC(\calG)$.
        For each  $i\in [\ell-1]$ let $\bde_{p_i}^{\bfB}=\bdb^{i+1}-\bdb^{i}$.   
        We call 
        $p_i$
        the \emph{$i$-th position change} of $\bfB$.  Whenever there is no ambiguity, we write $\bde_{p_i}$ for $\bde_{p_i}^{\bfB}$. 
      \item \label{Delta count} For each $\bfB=[\bdb^1, \ldots, \bdb^{\ell}] \in\MC(\calG)$ we call $P^{\bfB}\coloneqq(p_1, \ldots, p_{\ell-1})$ the \emph{tuple of the position changes in} $\bfB$, where $\bde_{p_i} = \bdb^{i+1}-\bdb^i$ for all $i\in[\ell-1]$.     
      \item Let $\bfB=[\bdb^1, \ldots, \bdb^\ell]\in \MC(\calG)$ and $P^{\bfB} = (p_1, \ldots, p_{\ell-1})$. For $s\in[2, \ell-1]$, we say that $\bdb^s \in \bfB$ satisfies the \emph{LQ condition} if $p_{s-1}<p_{s}$.
      \item Recall that $\sigma$ is the graded reverse lexicographic order on $T$. For any $\bfA= [\bda^1, \ldots, \bda^{\ell}] \in \MC(\calG)$, we denote $T_{\bfA} \coloneqq T_{\bda^1} \cdots T_{\bda^{\ell}}$. By abuse of notation, for any $\bfA, \bfB \in \MC(\calG)$ we let $\bfA>_{\sigma} \bfB$ if and only if $T_{\bfA}>_{\sigma} T_{\bfB}$. 
   \end{enumerate} 
\end{Definition}

\begin{Remark}
\label{PL}
    Adopt \Cref{LadderSetting} and let $\ell=\ell(M)$. 
    \begin{enumerate}[a]
        \item For every $\bfB=[\bdb^1, \ldots, \bdb^{\ell}] \in\MC(\calG)$ with $P^{\bfB}=(p_1, \ldots, p_{\ell-1})$ as in \Cref{positions}, $\{p_1,\ldots, p_{\ell-1}\}$ coincides with the collection of integers
      $$ P_{\calL \times [r]}\coloneqq\{ \underbrace{1, \ldots, 1}_{\Delta_1 \text{ times}}, \ldots, \underbrace{n, \ldots, n}_{\Delta_n \text{ times}}, \underbrace{n+1, \ldots, n+1}_{r-1 \text{ times}} \},  $$
    where for all $i\in[n]$ the integer $i$ appears with multiplicity $\Delta_i=v_i-u_i$, while $n+1$ appears $r-1$ times. Indeed, as observed in \Cref{cliqueproperties}~\ref{rem:first and last in every clique}, $\bdb^1=\bdu$, $\bdb^{\ell}=\bdv$. Moreover, by definition each $p_i \in [n+1]$ and for every $i\in[2, \ell-1]$, we have $\bdb^i_j\in [u_j,v_j]$ for all $j\in [n]$ and $\bdb^{i}_{n+1}\in [r]$. Additionally, $\ell = \sum \limits_{i=1}^{n}\Delta_{i}+r$ by \Cref{analyticSpread}.

    \item  
  Suppose $\bfA,\bfB\in\MC(\calG)$ with $\bfA>_{\sigma} \bfB$. Since $\sigma$ is the graded reverse lexicographic order, then by \Cref{cliqueproperties}~\ref{rem:first and last in every clique} there exists $i\in [2, \ell-1]$ so that $\bda^{i}\neq \bdb^{i}$ and $\bda^{j}=\bdb^{j}$ for all $j\in [i+1, \ell]$. Write $\bda^{i+1}-\bda^i=\bde_{q_i}$ and $\bdb^{i+1}-\bdb^i=\bde_{p_i}$, see \Cref{lem:conse}. Since $\bfA>_{\sigma} \bfB$, then $\bda^i>_{\tau}\bdb^i$.
   In fact, whenever $\bda^{i+1}=\bdb^{i+1}$ for some $i\in [\ell-1]$, then $\bda^i>_{\tau}\bdb^i$ if and only if $q_i<p_i$. Indeed, $\bda^i>_{\tau}\bdb^i$ if and only if the left most nonzero entry of $\bda^{i}-\bdb^{i}=-\bde_{q_i}+\bde_{p_i}$ is negative and this is equivalent to $q_i<p_i$.
    \end{enumerate}
\end{Remark}

We will see in \Cref{Thm: Linear Quotients} that, in order for $\ini_{\sigma}(\calJ)^{\vee}$ to have linear quotients, it suffices to prove that, given maximal cliques $\bfA,\bfB$ with $\bfA >_{\sigma} \bfB$, there exists another maximal clique $\bfC$ such that $|\bfC \setminus \bfB| =1$, $\bfC>_{\sigma} \bfB$ and $\bfB \setminus \bfC  \subseteq \bfB \setminus \bfA$.
We first show that the LQ condition as in \Cref{positions} is equivalent to the existence of such a maximal clique $\bfC$. 

\begin{Proposition}\label{LQIfandOnlyif}
Adopt \Cref{LadderSetting}. Let $\ell=\ell(M)$ and $\bfB\in \MC(\calG)$. For any $s\in[2, \ell-1]$, $\bdb^s\in\bfB$ satisfies the LQ condition if and only if there exists a $\bfC\in \MC(\calG)$ such that $\bfC>_{\sigma}\bfB$ and $\bfB \setminus \bfC=\{\bdb^s\}$.
\end{Proposition}

\begin{proof}
  Let $\bfB=[\bdb^1, \ldots, \bdb^\ell]$ and let $P^{\bfB}=(p_1, \ldots, p_{\ell-1})$. Suppose $\bdb^s$ satisfies the LQ condition for some $s\in [2, \ell-1]$. Then $\bdb^{s}-\bdb^{s-1}=\bde_{p_{s-1}}$, $\bdb^{s+1}-\bdb^{s}=\bde_{p_{s}}$, and  $p_{s-1}<p_s$. Let $\bdc \coloneqq \bdb^{s-1}+\bde_{p_s}$ and let $\bfC\coloneqq (\bfB\setminus\{\bdb^{s}\}) \cup \{\bdc\}$. We claim that $\bfC\in \MC(\calG)$ and $\bfC>_{\sigma} \bfB$.

We first show that $\bdc \in \calL \times[r]$, that is, for all $i\in [n]$ we must have $\bdc_i\in S_i$, $\bdc_{n+1}\in [r]$, and  $\bdc_i < \bdc_{i+1}$ for all $i \in [n-1]$. To this end, notice that  
$$ \bdb^{s+1}=\bdb^{s}+\bde_{p_s}=\bdb^{s-1}+\bde_{p_{s-1}}+\bde_{p_{s}}=\bdc+\bde_{p_{s-1}}. $$ 
For any $i\in[n+1]$, we have $\bdc_i=\bdb_i^{s-1}$ if $i\neq p_s$ and $\bdc_{p_s}=\bdb^{s+1}_{p_s}$. Then $\bdc_i\in S_i$ for all $i\in [n]$ and $\bdc_{n+1}\in [r]$. Thus we only need to show that $\bdc_i<\bdc_{i+1}$ for all $i\in[n-1]$.

Notice that for all $i\not\in\{ p_{s}, p_{s-1}-1\}$ we have $\bdc_i=\bdb_i^{s-1}<\bdb_{i+1}^{s-1}\le \bdb^{s+1}_{i+1}=\bdc_{i+1}$, as $i+1\neq p_{s-1}$.  On the other hand, $\bdc_{p_{s}}=\bdb^{s+1}_{p_{s}}<\bdb^{s+1}_{p_{s}+1}=\bdc_{p_s+1}$, since $p_{s-1}<p_s$. Moreover, $\bdc_{p_{s-1}-1}=\bdb_{p_{s-1}-1}^{s+1}=\bdb_{p_{s-1}-1}^{s-1}<\bdb_{p_{s-1}}^{s-1}=\bdc_{p_{s-1}}$, where the second equality follows since $p_{s-1}<p_s$. Thus, $\bdc\in \calL\times [r]$.

Now, it is clear that $\bdb^{s-1}>_{\tau} \bdc$ and $\bdc=\bdb^{s-1}+\bde_{p_s}>_{\tau} \bdb^{s-1}+\bde_{p_{s-1}} = \bdb^{s} >_{\tau}\bdb^{s+1}$, since $p_{s-1}<p_{s}$. Moreover, $\min\{\bdb^{s-1}, \bdc\}=\bdb^{s-1}$ since $\bdc=\bdb^{s-1}+\bde_{p_s}$ and for all $j\in [s-1]$ we have $\min\{\bdb^j, \bdc\}=\bdb^{j}$, since $\min\{\bdb^j, \bdb^{s-1}\}=\bdb^{j}$. Therefore, $\{\bdb^j, \bdc\}\in E(\calG)$ for all $j\le s-1$. Similarly, $\{\bdc, \bdb^j\}\in E(\calG)$ for all $j\ge s+1$. Hence $\bfC=[\bdb^1, \ldots, \bdb^{s-1}, \bdc, \bdb^{s+1}, \ldots, \bdb^{\ell}] \in \MC(\calG)$. Finally, since $\bfC\setminus \bfB=\{\bdc\}$, $\bfB\setminus \bfC=\{\bdb^{s}\}$, and $\bdc>_{\tau}\bdb^{s}$, we deduce that $T_{\bfC} >_{\sigma}T_{\bfB}$, hence $\bfC>_{\sigma} \bfB$.

Conversely, suppose that $\bfC \in \MC(\calG)$ is such that $\bfC>_{\sigma}\bfB$ and $\bfB \setminus\bfC=\{\bdb^s\}$ for some $s\in[2, \ell-1]$. We prove that $\bdb^s$ satisfies the LQ condition. 

Since $\bfB \setminus\bfC=\{\bdb^s\}$, by \Cref{analyticSpread} there is some $\bdc$ such that $\bfC\setminus \bfB=\{\bdc\}$. We claim that  $\bfC=[\bdb^1, \ldots, \bdb^{s-1}, \bdc, \bdb^{s+1}, \ldots, \bdb^{\ell}]$. Indeed, if either $\bdc>_{\tau} \bdb^{s-1}$ or $\bdc<_{\tau} \bdb^{s+1}$, then $\bdb^{s-1}, \bdb^{s+1}$ will be consecutive tuples in $\bfC$, which is impossible by \Cref{lem:conse} since $\bdb^{s+1}=\bdb^{s-1}+\bde_{p_{s-1}}+\bde_{p_{s}}$.
This shows that $\bfC=[\bdb^1, \ldots, \bdb^{s-1}, \bdc, \bdb^{s+1}, \ldots, \bdb^{\ell}]$.

Now, let $P^{\bfC}=(q_1, \ldots, q_{\ell-1})$ and recall that $\{p_1, \ldots, p_{\ell-1}\} = \{q_1, \ldots, q_{\ell-1}\}$ by \Cref{PL}. Moreover, $p_i=q_i$ for all $i\in [s-2] \cup [s+1, \ell-1]$. Therefore, $\{p_{s-1}, p_s\}=\{q_{s-1}, q_s\}$, and thus $p_{s-1}=q_s$ and $p_s=q_{s-1}$. Since $\bfC >_{\sigma} \bfB$, then by \Cref{PL} $\bdc >_{\tau} \bdb^{s}$. Thus, $q_s<p_s$, or in other words, $p_{s-1}<p_s$. In particular, $\bdb^s$ satisfies the LQ condition.
\end{proof}

We now prove the main theorem of this section.

\begin{Theorem} \label{Thm: Linear Quotients}
   Adopt \Cref{LadderSetting}. The ideal $(\ini_{\sigma}(\calJ))^{\vee}$ has linear quotients.
\end{Theorem}

\begin{proof}
Let $\MC(\calG)$ be the set of maximal cliques of the graph $\calG$. As observed in \Cref{rmk:maximal_cliques}, we can write the minimal generators of $(\ini_{\sigma}(\calJ))^\vee$ as
$$ G((\ini_{\sigma}(\calJ))^\vee) = \Set{T_{\bfA^\complement}\coloneqq \prod_{\bda \in \calL \times [r] \setminus \bfA}T_{\bda} \text{ such that } \bfA\in \MC(\calG)}. $$
Let $\mu=\mu((\ini_{\sigma}(\calJ))^\vee)$ and write $G((\ini_{\sigma}(\calJ))^\vee)= \{T_{\bfA_1^{\complement}}, \ldots, T_{\bfA_{\mu}^{\complement}}\}$, where  $\bfA_{i} >_{\sigma} \bfA_{i+1}$ for all $i\in[\mu-1]$ and $\sigma$ is as in \Cref{positions}.  For each $i\in[2, \mu]$, let $Q_{i}\coloneqq \braket{T_{{\bfA^\complement_1}}, \ldots, T_{\bfA_{i-1}^{\complement}}}:T_{\bfA_i^{\complement}}$. Then
$$ Q_{i}=\sum \limits_{j=1}^{i-1}\braket{T_{\bfA_j^{\complement}}:T_{\bfA_i^{\complement}}}=\sum \limits_{\bfA_j>_{\sigma} \bfA_{i}}\braket{T_{\bfA_j^{\complement}}:T_{\bfA_i^{\complement}}}.$$
Notice that 
$$ \braket{T_{\bfA_j^\complement}:T_{\bfA_i^\complement}}= \left\langle\frac{\prod_{\bda \in \calL \times [r]}T_{\bda}}{\prod_{\bda^h \in \bfA_j}T_{\bda^h}}:\frac{\prod_{\bda \in \calL \times [r]}T_{\bda}}{\prod_{\bda^k \in \bfA_i}T_{\bda^k}}\right\rangle= \Bigl\langle
    \prod_{\bda^h \in \bfA_i\setminus \bfA_j}T_{\bda^h} \Bigr \rangle. $$
Therefore, $Q_i= \Bigl \langle \prod \limits_{\bda^h \in \bfA_i\setminus \bfA_j} T_{\bda^h} : \bfA_j>_{\sigma} \bfA_i \Bigr \rangle$. To prove that $(\ini_{\sigma}(\calJ))^{\vee}$ has linear quotients, we need to show that the ideals $Q_i$ are generated by variables. Equivalently, we must show that for any maximal clique $\bfB$, the ideal   
$$ Q_{\bfB}=\sum_{\substack{\bfA>_{\sigma} \bfB \\ \bfA\in \MC(\calG)}} \braket{T_{\bfA^\complement}: T_{\bfB^\complement}}=
\Bigl \langle\prod_{\bdb^i \in \bfB\setminus \bfA}T_{\bdb^i} : \bfA>_{\sigma} \bfB \Bigr \rangle $$
is generated by variables. 
For this, it suffices to prove that for any $\bfA,\bfB \in \MC(\calG)$ with $\bfA>_{\sigma} \bfB$ we can find a suitable maximal clique $\bfC$ with $\bfC>_{\sigma} \bfB$ such that $|\bfB\setminus \bfC|=1$ and $\bfB\setminus \bfC \subseteq \bfB\setminus \bfA$.

   Let $\bfA,\bfB \in \MC(\calG)$ with $\bfA>_{\sigma} \bfB$.  Then by \Cref{PL} there exists $\delta_2 \in [2,\ell-1]$ such that $\bda^k=\bdb^k$ for all $k\in [\delta_2+1,\ell]$ and $\bda^{\delta_2}>_{\tau}\bdb^{\delta_2}$.  Let $\delta_1=\max \{i<\delta_2 : \bda^{i} =\bdb^{i}\}$.  Then $\bda^{\delta_1}=\bdb^{\delta_1}$, $\bda^{\delta_2+1}=\bdb^{\delta_2+1}$, and  $\bda^{k}\neq \bdb^k$ for all $k\in [\delta_1+1,\delta_2]$.  
    Write $P^{\bfA}=(q_1,\ldots, q_{\ell-1})$ and $P^{\bfB}=(p_1, \ldots, p_{\ell-1})$.  Therefore, 
\begin{eqnarray*}
    \bda^{\delta_2+1}-\bda^{\delta_1}&=&(\bda^{\delta_2+1}-\bda^{\delta_2})+(\bda^{\delta_2}-\bda^{\delta_2-1})+\cdots +(\bda^{\delta_1+1}-\bda^{\delta_1})=\sum\limits_{k=\delta_1}^{\delta_2} \bde_{q_k}, 
\end{eqnarray*}
and similarly, $\bdb^{\delta_2+1}-\bdb^{\delta_1}=\sum\limits_{k=\delta_1}^{\delta_2} \bde_{p_k}$. Since $\bda^{\delta_2+1}-\bda^{\delta_1} = \bdb^{\delta_2+1}-\bdb^{\delta_1}$ we must have $\{q_{\delta_1},\ldots,q_{\delta_2}\} = \{p_{\delta_1},\ldots,p_{\delta_2}\}$. Choose $w=\min \{p_{i}: i\in [\delta_1, \delta_2]\}$
 and let $v=\max\{j : \bde_{p_j}=\bde_w, \, j\in [\delta_1,\delta_2]\}$. We claim that $\bdb^{v+1}$ satisfies the LQ condition. 

Notice that since $\bda^{\delta_2+1}=\bdb^{\delta_2+1}$ and $\bda^{\delta_2}>_{\tau} \bdb^{\delta_2}$, then $q_{\delta_2}<p_{\delta_2}$, by \Cref{PL}. Hence $w\neq p_{\delta_2}$ and thus $v<\delta_2$. Therefore, by construction $p_v<p_{v+1}$ and thus $\bdb^{v+1}$ satisfies the LQ condition as in \Cref{positions}. Hence by \Cref{LQIfandOnlyif} there exists $\bfC\in \MC(\calG)$ such that $\bfC>_{\sigma} \bfB$ and $\bfB\setminus \bfC=\{\bdb^{v+1}\}$. Moreover, $\bfB\setminus \bfC= \{\bdb^{v+1}\} \subseteq \bfB\setminus \bfA$, since $v+1\le \delta_2$.
\end{proof}

We finally obtain a description of the generators of the linear quotient ideals. 

\begin{Corollary} \label{cor: linear quot}
    Adopt \Cref{LadderSetting} and let $\ell=\ell(M)$. Let $\bfB= [\bdb^1, \ldots, \bdb^{\ell}]\in \MC(\calG)$ and consider the linear quotient ideal $Q_{\bfB}=\sum_{\bfA>_{\sigma} \bfB} \braket{T_{\bfA^\complement}: T_{\bfB^\complement}}$. Then 
    $$Q_{\bfB}=\langle T_{\bdb^{s}} : \bdb^{s} \mbox{ satisfies the } LQ \mbox{ condition}, \,s\in[2, \ell-1]\rangle.$$
\end{Corollary}

\begin{proof}
    The result follows from \Cref{LQIfandOnlyif} and the proof of \Cref{Thm: Linear Quotients}.
\end{proof}

\section{The sequence index of a maximal clique and the regularity of $\calF(M)$}\label{sec:reg}

In this section, we determine the regularity of the special fiber ring $\calF(M)$, where $M=\bigoplus_{i=1}^{r}L$ is as in \Cref{LadderSetting}.
As we discussed in \Cref{SAGBIandGB}~\ref{regToPd}, this task is reduced to calculating the projective dimension of $(\ini_{\sigma}(\calJ))^{\vee}$. Since by \Cref{cor: linear quot} $(\ini_{\sigma}(\calJ))^{\vee}$ has linear quotients, then its projective dimension is the maximal number of minimal generators of the linear quotient ideals $Q_{\bfB}$ with $\bfB\in \MC(\calG^r)$; see \cite[Corollary 8.2.2]{HHMon}. 
As it turns out, in order to compute this maximum we will need to compare maximal cliques of the graphs $\calG^r$ and $\calG^1$ as in \Cref{LadderSetting}. Hence, throughout the section, we will always retain the appropriate superscripts in our notation.

We begin by providing an algorithm to construct a special maximal clique of $\calG^r$ whose linear quotient ideal will have the maximal possible number of minimal generators; see \Cref{reg}.

\begin{Construction}\label{MaxLengthClique}
    Adopt \Cref{LadderSetting}. Let $\bda^1 \coloneqq \bdu= (u_1,\ldots,u_{n},1)\in \calL \times [r]$ and $H_0 \coloneqq 0$. 
  For each $k \ge 1$:
       \begin{enumerate} 
       \item Define \begin{align*}
          P_k \coloneqq & \left\{q \in [n-1] : \bda^{1+H_{k-1}}_{q+1}-\bda^{1+H_{k-1}}_{q}>1, \, \bda^{1+H_{k-1}}_{q}< v_q \right\} \\ &\cup \left \{n : \bda^{1+H_{k-1}}_{n}<m \right \}\cup \left\{n+1: \bda^{1+H_{k-1}}_{n+1}<r \right\}.
      \end{align*}
      If $P_k=\emptyset$, then the process stops.  If $P_k\neq \emptyset$, then write $P_k=\{p_{k,1}<p_{k,2}<\cdots <p_{k,{|P_k|}}\} $ and set $H_{k} \coloneqq \sum_{s=1}^{k} |P_s|$. 
      \item  For each $i\in [|P_k|],$ let $\bda^{i+1+H_{k-1}}=\bda^{i+H_{k-1}}+\bde_{p_{k,i}}$.
 \end{enumerate}

This process will terminate as $\bda^1 \in \calL \times [r]$ and $\calL\times [r]$ is a finite set. Let $\alpha \coloneqq \max\{k: P_k\neq \emptyset\}$ and write $\calA^r\coloneqq [\bda^1, \ldots, \bda^{\lambda}]$, where $\lambda=H_{\alpha}+1$. Let  $P^{\calA^r}=(P_1, \ldots, P_{\alpha})$. 

\end{Construction}

\begin{Lemma} 
Let $\calA^r=[\bda^1, \ldots, \bda^{\lambda}]$ be as in \Cref{MaxLengthClique}. Then $\calA^r\in \MC(\calG^r)$ and in particular, $\lambda=\ell(M)$. 
\end{Lemma}

\begin{proof} 
Let $\alpha=\max\{k: P_{k}\neq \emptyset\}$ and $\lambda=H_{\alpha}+1$ as in \Cref{MaxLengthClique}. We first claim that $\bda^{\lambda}=\bdv =(v_1, \ldots, v_n,r)$ as in \Cref{cliqueproperties}~\ref{rem:first and last in every clique}. Since $P_{\alpha+1}=\emptyset$, then  
$\bda^{\lambda}_{n+1}=r$, $\bda^{\lambda}_n=m=v_n$. 
Suppose $\bda^{\lambda}\neq \bdv$. Then there exists $q\in[n-1]$ such that $\bda^{\lambda}_q\neq v_q$. By \Cref{MaxLengthClique} we must have that $\bda^{\lambda}_{q+1}-\bda^{\lambda}_q=1$ and $\bda^{\lambda}_q<v_q$.  
Therefore, $\bda^{\lambda}_{q+1}=\bda^{\lambda}_q+1<v_q+1\le v_{q+1}$. Since $P_{\alpha+1}=\emptyset$, then $\bda^{\lambda}_{q+2}-\bda^{\lambda}_{q+1}=1$ as $\bda^{\lambda}_{q+1}<v_{q+1}$. Iterating the previous argument, we eventually 
obtain that $\bda^{\lambda}_{n}<v_n$, which is a contradiction as $\bda^{\lambda}_n=v_n$. Thus $\bda^{\lambda}=\bdv$. 

Now, notice that by construction it follows that $\bda^i_p\le \bda^j_p$ for all $i\le j$, $i,j \in [\lambda]$, and for all $p\in [n+1]$. Therefore, by \Cref{cliqueproperties}~\ref{rem:tuple inequalities},
$\calA^r$ is a clique of $\calG^r$. It remains to show that it is a maximal clique or equivalently, that $\lambda=\ell(M)$; see \Cref{analyticSpread}.  Since $\calA^r$ is a clique, then $\lambda \le \ell(M)$.
Suppose $\lambda<\ell(M)$.  Then we can find a collection of $\ell(M)-\lambda$ tuples $\bdb^j \in \calL\times [r]$ that can be added to $\calA_{\calL\times [r]}$ to produce a maximal clique $\bfB$. Therefore, without loss of generality there exist $k\in [\ell(M)-\lambda]$ and $i\in[\lambda]$ such that $ \bda^i>_{\tau} \bdb^1 >_{\tau} \ldots >_{\tau} \bdb^k >_{\tau} \bda^{i+1}$ with $\bda^{i}, \bdb^1, \ldots, \bdb^k, \bda^{i+1} \in  \bfB$. Then by \Cref{lem:conse} there exist $q, q_1, \ldots, q_{k} \in [n+1]$ such that  $\bdb^{1}=\bda^i+\bde_{q}$, $\bdb^{j+1}=\bdb^j+\bde_{q_{j}}$ for all $j\in [k-1]$, and  $\bda^{i+1}=\bdb^{k}+\bde_{q_{k}}$. On the other hand, by \Cref{MaxLengthClique} there exists $p \in [n+1]$ such that $\bda^{i+1}=\bda^i+\bde_p$. In particular,  $$\bde_p=\bda^{i+1}-\bda^{i}=(\bda^{i+1}-\bdb^k)+(\bdb^{k}-\bdb^{k-1})+\ldots +(\bdb^{1}-\bda^{i})=\sum_{i=1}^{k}\bde_{q_i}+\bde_q,$$
which is impossible since $k\ge 1$.
\end{proof}

The sets $P_k$ appearing in \Cref{MaxLengthClique} motivate the following definition.

\begin{Definition}\label{si Index} 
Adopt \Cref{LadderSetting}. 
   For each $\bfB\in \MC(\calG^r)$, we write the tuple of position changes $P^{\bfB}\coloneqq(P_1^{\bfB}, \ldots, P_{\beta}^{\bfB})$, where for each $i\in [\beta]$
   $$P_i^{\bfB}\coloneqq\{p_{i,1}<p_{i,2}<\cdots < p_{i, |P_i^{\bfB}|}\}$$ 
  with $p_{k,|P_k^{\bfB}|}\ge p_{k+1,1}$ for all $k\in[\beta-1]$. In other words, each $P_i^{\bfB}$ is a maximal increasing sequence. 
   For each $\bfB\in \MC(\calG^r)$, let the \emph{sequence index} of $\bfB$ be $\si{\bfB}\coloneqq\beta$, that is, $\si{\bfB}$ is the number of maximal increasing sequences in $P^{\bfB}$. 
\end{Definition}

The next example illustrates \Cref{MaxLengthClique} and \Cref{si Index}.

\begin{Example}\label{ex:Construction}
Consider the ladder matrix in \Cref{LadderMatrix} and let $r=2$. Recall that $S_1=[1,5]=[u_1,v_1]$, $S_2=[3,6]=[u_2,v_2]$ and $S_3=[4,9]=[u_3,v_3]$. As computed in \Cref{ex:analyticspread}, we have $\ell=\ell(M)=14$. By \Cref{PL}, we have $P_{\calL \times [r]}=\{1,1,1,1,2,2,2,3,3,3,3,3,4\}$. 
Following \Cref{MaxLengthClique}, let $\bda^1=(1,3,4,1)$ and $H_0=0$. Then $P_1=\{1,3,4\}$ as $u_3-u_2=1$, and $H_1=|P_1|=3$.
    We obtain 
    $$\bda^2=(2,3,4,1),\,\bda^3=(2,3,5,1),\,\bda^4=(2,3,5,2).$$ 
    Next we have $P_2=\{2,3\}$ and $H_2=\sum_{s=1}^2|P_s|=5$.  Notice that we do not include $4$ in $P_2$ as $\bda_2^4=2=r$.
    We obtain 
    $$\bda^5=(2,4,5,2),\,\bda^6=(2,4,6,2).$$ 
    Now we have $P_3=\{1,2,3\}$, $H_3=\sum_{s=1}^3|P_s|=8$, and we obtain $$\bda^7=(3,4,6,2),\,\bda^8=(3,5,6,2),\,\bda^{9}=(3,5,7,2).$$ 
    Continuing on we have $P_4=\{1,2,3\}$ and $H_4=\sum_{s=1}^4|P_s|=11$. We obtain
    $$\bda^{10}=(4,5,7,2),\, \bda^{11}=(4,6,7,2),\, \bda^{12}=(4,6,8,2).$$ 
    Finally, we have $P_5=\{1,3\}$ and $H_5=\sum_{s=1}^5|P_s|=13$. Notice that we do not include $2$ as $\bda^{12}_2=6=v_2$. We obtain,
    $$\bda^{13}=(5,6,8,2),\,\bda^{14}=(5,6,9,2)=\bdv.$$ 
    It is clear that $P_6=\emptyset$. We let $\calA^2=[\bda^1,\ldots,\bda^{14}]$, $\{P_1,\ldots,P_5\}=P_{\calL \times [2]}$, $\si{\calA^2}=5$, and $P^{\calA^2}=(1,3,4,2,3,1,2,3,1,2,3,1,3)$.
\end{Example}

In the following proposition we prove that, for any maximal clique $\bfB\in \MC(\calG^r)$, the sequence index $\si{\bfB}$ allows us to calculate the number of minimal generators of the linear quotient ideal $Q_{\bfB}$ as in \Cref{cor: linear quot}.

\begin{Proposition}\label{pdLowA}
Adopt \Cref{LadderSetting} and let $\ell= \ell(M)$. For any  maximal clique $\bfB\in\MC(\calG^r)$ with $P^{\bfB}=(P_1^{\bfB}, \ldots, P_{\si{\bfB}}^{\bfB})$, we have $\mu(Q_{\bfB})=\sum_{s=1}^{\si{\bfB}} (|P_{s}^{\bfB}|-1)=\ell-1-\si{\bfB}$. 
\end{Proposition}

\begin{proof} 
Write $\bfB=[\bdb^1, \ldots, \bdb^{\ell}]$ and let $P_k=P_k^{\bfB}=\{p_{k,1}< \cdots < p_{k, |P_k|}\}$ for all $k\in [\beta]$, where $\beta=\si{\bfB}$. Let $J_0=0$ and $J_k=\sum_{s=1}^{k}|P_s|$. 
 Fix $k\in[\beta]$ and let $i\in[2,|P_k|]$. Then
    \begin{align*}
        \bdb^{i+J_{k-1}}=\bdb^{i-1+J_{k-1}}+\bde_{p_{k,{i-1}}}\quad \text{ and } \quad \bdb^{i+1+J_{k-1}}=\bdb^{i+J_{k-1}}+\bde_{p_{k,{i}}}. 
    \end{align*} 
    Since $p_{k,{i-1}}<p_{k,{i}}$, then $\bdb^{i+J_{k-1}}$ satisfies the LQ condition as in \Cref{positions}.
 On the other hand, for each  $k\in [\beta]$, $\bdb^{1+J_k}$ does not satisfy LQ condition as $p_{{k-1},{|P_{k-1}|}}\ge p_{k,1}$ by the definition of strictly increasing sequences. Hence we have exactly $\sum_{k=1}^{\beta} (|P_{k}|-1)$ elements that satisfy the LQ condition. Therefore by \Cref{cor: linear quot}, 
 $$Q_{\bfB}=\langle T_{\bdb^{i+J_{k-1}}}: k\in [\beta], i\in[2, |P_k|]\rangle$$
 and hence 
by \Cref{positions}~\ref{Delta count}, 
\[ \mu(Q_{\bfB})=\sum_{k=1}^{\beta} \big(|P_{k}|-1 \big) = \sum_{k=1}^{\beta} |P_{k}|- \beta = |P^{\bfB}| - \beta = \ell -1 - \si{\bfB}. \qedhere \]
\end{proof}

Our next goal is to show that $\max \{ \mu(Q_{\bfB}) : \bfB \in \MC(\calG^r)\}$ is in fact $\mu(Q_{\calA^r})$. Equivalently, we will prove that $\calA^r$ has the smallest possible sequence index. To achieve this, a key step will be to combine distinct maximal cliques to create a new maximal clique, while keeping the sequence indices under control. The following remark provides several criteria for a sequence of $\ell$-tuples to form a maximal clique.

\begin{Remark}\label{proving Strategy} 
Adopt \Cref{LadderSetting}. Let $\ell=\ell(M)$ and let $P=(p_1,\ldots, p_{\ell-1})$ be a tuple such that $\{p_1,\ldots,p_{\ell-1}\}$ coincides with $P_{\calL \times [r]}$.
Let $\bdb^1=\bdu=(u_1,\ldots, u_n,1)$ and for all $i\in [\ell-1]$ set $\bdb^{i+1}=\bdb^{i}+\bde_{p_i}$. Then, for $\bfB=[\bdb^1, \ldots, \bdb^{\ell}]$, the following are equivalent: 
\begin{enumerate}[a]
    \item $\bfB\in \MC(\calG^r)$ with $P^{\bfB}=P$;
    \item for all $i \in [\ell-1]$, $\bdb^{i+1}\in \calL \times [r]$;
    \item for all $i \in [\ell-1]$: if $p_i=n+1$, then $\bdb^{i}_{p_i}<r$; if $p_i=n$, then $\bdb^{i}_{p_i}<v_n$; and if $p_i<n$, then $\bdb^{i}_{p_i+1}-\bdb^{i}_{p_i}>1$ and $\bdb^{i}_{p_i}< v_{p_i}$; 
    \item \label{jth i-1 after jth i}  for all $i\in[2,n]$ and all $j\in [\Delta_{i-1}-\epsilon_{i-1}+1]$, the $j$-th entry of $i$ in $P$ has to be placed before the $(j+\epsilon_{i-1}-1)$-th entry of $i-1$.
\end{enumerate}
 In regards to part \ref{jth i-1 after jth i}, write  $P=(P_1, \ldots, P_{\beta})$ so that each $P_i$ is a maximal increasing sequence of elements in $P$ similarly as in \Cref{si Index}. For every $i \in [2,n]$, write 
 $\{k \in [\beta]: i\in P_{k} \}=\{k_1 < \cdots < k_{\Delta_i}\}$ and let $\{s \in [\beta]: i-1\in P_{s}\}=\{s_1< \cdots< s_{\Delta_{i-1}}\} $.
 Then $\bfB $ is a maximal clique if and only if $s_{j+\epsilon_{i-1}-1}> k_j$ for all $j\in [\Delta_{i-1}-\epsilon_{i-1}+1]$, since $\Delta_{i-1}\le \Delta_i+\epsilon_{i-1}-1$.
\end{Remark}

We now prove that an arbitrary maximal clique $\bfB$ can be combined in an appropriate way with the maximal clique $\calA^r$ defined in \Cref{MaxLengthClique} to produce another maximal clique $\bfC$ with $\si{\bfC} \le \si{\bfB}$, as we will see in the proof of \Cref{best reg ind}.

\begin{Lemma} \label{lem:shift left}
   Adopt \cref{LadderSetting} and let $\ell=\ell(M)$. 
   Let $\calA^r$ be the maximal clique defined in \Cref{MaxLengthClique} and let $\bfB \in \MC(\calG^r)$ with $\bfB \neq \calA^r$. Write $P^{\calA^r}=(p_1, \ldots, p_{\ell-1})$ and $P^{\bfB}=(q_1, \ldots, q_{\ell-1})$.
   Let $d=\min \{k : p_k\neq q_k\}$ and let $j=\min\{i>d: q_i=p_d \}$.   
   Then there exists a maximal clique $\bfC \in \MC(\calG^r)$ so that 
     $$P=(p_1, \ldots, p_{d-1}, p_d, q_d, \ldots, \widehat{q_{j}}, \ldots, q_{\ell-1}) = P^{\bfC}.$$
\end{Lemma}

\begin{proof}
Using the setting as in the statement, notice that since $p_d \in P_{\calL \times [r]}$ with multiplicity $\Delta_{p_d}$, there must exist $i>d$ such $q_i=p_d$. 
Then $j=\min\{i>d: q_i=p_d\}$ is well defined. 

Let $\bfC=[\bdc^1, \ldots, \bdc^{\ell}]$, where $\bdc^1=\bdu$ and 
$$\bdc^{i+1}= \begin{cases}
\bdc^{i}+\bde_{p_i} & \text{ if } i\in [d] \\
\bdc^{i}+\bde_{q_{i-1}} & \text{ if } i\in [d+1,j]\\
\bdc^{i}+\bde_{q_{i}} & \text{ if } i\in [j+1,\ell-1].
\end{cases}$$
We claim that $\bfC$ is a maximal clique of $\calG^r$ with $P^{\bfC}=P$. If $P=P^{\calA^r}$, then $\bfC=\calA^r$ and we are done. Hence, we may assume that $P \neq P^{\calA^r}$. 
Write $\calA^r=[\bda^1, \ldots, \bda^{\ell}]$ and $\bfB=[\bdb^1, \ldots, \bdb^{\ell}]$.
It is clear that $d \geq 2$, since $\bda^1=\bdb^1=\bdu$, as in \Cref{cliqueproperties}. Moreover, by the choice of $d$ we have $\bda^i=\bdb^i=\bdc^{i} \in \calL\times [r]$ for all $i\le d$.  On the other hand, for all $i\in [j+1, \ell]$ we have $\bdc^{i}=\bda^{d}+\bde_{p_d}+\sum_{s=d}^{i-2}\bde_{q_{s}}=\bdb^{d}+\sum_{s=d}^{i-1}\bde_{q_{s}}=\bdb^{i}\in \calL\times [r]$. By \Cref{proving Strategy}, it only remains to show that $\bdc^{i} \in \calL \times [r]$ for all $i \in [d+1,j]$.

Let $i \in [d+1,j]$ and notice that $\bdc^{i}=\bdb^{i-1}+\bde_{p_d}$ and $\bdb^{i-1}_{p_d}=\bdb^d_{p_d}=\bda^d_{p_d}$. If $p_d=n+1$ or $p_d=n$, then   $\bdc^{i}\in \calL \times [r]$, since $\bdb^{i-1}_{p_d}=\bda^d_{p_d}$ and $\bda^{d+1}=\bda^{d}+\bde_{p_d}\in \calL\times[r]$. Suppose now that $p_d\neq n,n+1$. Then  $\bdb^{i-1}_{p_d+1}\ge \bdb^{d}_{p_d+1}=\bda^d_{p_d+1}$, $\bdb^{i-1}_{p_d}= \bdb^{d}_{p_d}=\bda^d_{p_d}$, and therefore, 
\begin{align*}
    \bdb^{i-1}_{p_d+1}-\bdb^{i-1}_{p_d}\ge \bdb^{d}_{p_d+1}-\bdb^{d}_{p_d}=\bda^{d}_{p_d+1}-\bda^{d}_{p_d}\ge 2 \quad \mbox{ and } \quad \bdb^{i-1}_{p_d}=\bda^{d}_{p_d}<v_{p_d},
\end{align*}
since $\bda^{d+1}=\bda^{d}+\bde_{p_d}\in \calL\times [r]$. 
Hence $\bdc^{i}\in \calL\times [r]$ and the proof is complete.
\end{proof}

Using the previous lemma, we can show that $\calA^r$ has the smallest possible sequence index.

\begin{Proposition}\label{best reg ind}
    Adopt \Cref{LadderSetting} and let $\calA^r$ be the maximal clique from \Cref{MaxLengthClique}.
    Then $\si{\calA^r}=\min\{\si{\bfB} : \bfB\in \MC(\calG^r)\}$.
\end{Proposition}

\begin{proof}
Let $\calA=\calA^r$, $\calG= \calG^r$ and $\bfB\in \MC(\calG)$. 
Let $\ell=\ell(M)$ and write $P^{\calA}=(P_1^{\calA}, \ldots, P_{\alpha}^{\calA})=(p_1, \ldots, p_{\ell-1})$ and $P^{\bfB}=(P_{1}^{\bfB}, \ldots, P_{\beta}^{\bfB})=(q_1, \ldots, q_{\ell-1})$  for the tuples of position changes for $\calA$ and $\bfB$, respectively.  Here $\si{\calA}=\alpha$ and $\si{\bfB}=\beta$. We claim that $\alpha\le \beta$. If $\alpha=\beta$, there is nothing to show. Hence we may assume that $\alpha\neq \beta$ and $\calA\neq \bfB$.
 
Let $\delta=\min\{k:P_k^{\calA}\neq P_{k}^{\bfB}\}$, $d=\min\{k:p_k\neq q_k\}$, and notice that $p_d\in P_{\delta}^{\calA}$, $p_d\not\in P^{\bfB}_{\delta}$. It is clear that $d \geq 2$, since $\bda^1=\bdb^1=\bdu$, as in \Cref{cliqueproperties}. Moreover, $\bda^i=\bdb^i$ for all $i\le d$, since $p_k=q_k$ for all $k\in[d-1]$. 
  It then follows that $P_{\delta}^{\bfB}\subsetneq P_{\delta}^{\calA}$. Indeed, suppose instead we can find $q\in P_{\delta}^{\bfB}$ such that $q\notin P_{\delta}^{\calA}$. Then, we have two cases to consider:
    \begin{enumerate}
        \item Suppose $q=n+1$. Since $q\not\in P_{\delta}^{\calA}$, then by the construction of $\calA$, we must have $\delta \ge r$ and $n+1$ appears $r-1$ times in $\{P_1^{\calA},\ldots,P_{\delta-1}^{\calA}\}=\{P_1^{\bfB},\ldots,P_{\delta-1}^{\bfB}\}$. This is a contradiction since this would imply that $n+1$ appears in $P^\bfB$  more than $r-1$ times. 
        \item Suppose $q\neq n+1$.  If $q=n$, then $q = q_j$ for some $j \geq d$, and hence $\bdb^{d}_n<v_n$. Since  $\bdb^{d}=\bda^{d}$, then by the construction of $\calA$ we must have $q=n\in P^{\calA}_{\delta}$, a contradiction. 
        Hence we may assume now that $q\in [n-1]$. Since $q\notin P^{\calA}_{\delta}$, we must have either $\bda^d_{q+1}-\bda^d_{q}=1$ or $\bda^d_{q}=v_{q}$. This is equivalent to $\bdb^d_{q+1}-\bdb^d_{q}=1$ or $\bdb^d_{q}=v_{q}$, which shows that $q\notin P^{\bfB}_{\delta}$, a contradiction.
        \end{enumerate}
Therefore, $P_{\delta}^{\bfB}\subsetneq P_{\delta}^{\calA}$ as claimed. Also, since $P_{\delta}^{\bfB}\neq \emptyset$ and $p_d \in  P_{\delta}^{\calA} \backslash P_{\delta}^{\bfB}$, we obtain that $|P_{\delta}^{\calA}|\ge 2$. 

Now, in order to prove that $\alpha \le \beta$, we iteratively construct a sequence of maximal cliques $\bfB_i$ such that $\si{\bfB_i}\le \beta$ for every $i$ and $\bfB_i= \calA$ for some $i$. 
First, by \Cref{lem:shift left} there exists a maximal clique $\bfB_1\in \MC(\calG)$ such that 
  $$P^{\bfB_1}=(p_1, \ldots, p_{d-1}, p_d, q_d, \ldots, \widehat{q_{j}}, \ldots, q_{\ell-1}),$$
where $j=\min\{i>d:q_i=p_d\}$. Write $P^{\bfB_1}=(P_1^{\bfB_1}, \ldots, P_{\beta_1}^{\bfB_1})$, where $\beta_1=\si{\bfB_1}$. We show that $\beta_1\le \beta$.

Notice that, by the definitions of $\delta$ and $d$, we have $P_{k}^{\bfB_1}=P_{k}^{\bfB}=P_{k}^{\calA}$ for all $k \in [\delta-1]$. 
Moreover, removing $q_j$ from $P^{\bfB}$ does not increase the number of strictly increasing sequences in $P^{\bfB_1}$. 
Now, if $p_d<q_d$, then $p_d, q_d\in P_{\delta}^{\bfB_1}$ and thus $\beta_1 \le \beta$. On the other hand, if $p_d>q_d$,  then $p_{d-1}\in P_{\delta}^{\bfB}$ and $q_d\in P_{\delta+1}^{\bfB}$, since  $P_{\delta}^{\bfB} \subsetneq P_{\delta}^{\calA}$ and $|P_{\delta}^{\calA}|\ge 2$. In particular, $p_{d-1},  p_d\in P_{\delta}^{\bfB_1}$ and $q_d\in P_{\delta+1}^{\bfB_1}$, whence $\beta_1\le \beta$ also in this case. 

If $\bfB_1=\calA$, then $\alpha=\beta_1 \le \beta$ as claimed. Otherwise,  $P^{\bfB_1}\neq P^{\calA}$ and by \Cref{lem:shift left} and the argument above we can construct $\bfB_2\in \MC(\calG)$ with $\si{\bfB_2}\le \si{\bfB_1}$. If $\bfB_2=\calA$, then $\alpha=\si{\bfB_2}\le \beta$. Otherwise, we keep applying \Cref{lem:shift left} and the argument above to $\bfB_i$ and $\calA$ to produce $\bfB_{i+1}$ with $\si{\bfB_{i+1}}\le \si{\bfB_i}$. 
Notice that the first index where the entries of $P^{\bfB_1}$ and $P^{\calA}$ are not equal will be strictly bigger than $d$. 
This guarantees that the above process will terminate, producing a maximal clique $\bfB_i=\calA$. Hence there exists an $i\ge 1$ such that $\alpha=\si{\bfB_i}\le \cdots \le \si{\bfB_1}\le \beta$. 
\end{proof}

The previous result reduces the calculation of $\max \{ \mu(Q_{\bfB}) : \bfB \in \MC(\calG^r)\}$ to computing $\si{\calA^r}$. In fact, it will be enough to compute $\si{\calA^1}$, where $\calA^1$ is the maximal clique of the graph $\calG^1$ obtained in \Cref{MaxLengthClique} with $r=1$. This is essentially due to following result, which relates the maximal cliques of $\calG^r$ and those of $\calG^1$.

\begin{Proposition}\label{addn+1}
 Adopt \Cref{LadderSetting}. 
 If $\bfB^r\in \MC(\calG^r)$, then there exists $\bfB^1\in \MC(\calG^1)$ so that $P^{\bfB^1}$ is obtained from $P^{\bfB^r}$ by removing all entries equal to $n+1$ in $P^{\bfB^r}$.  Conversely, for any $\bfB^1\in \MC(\calG^1)$, there exists $\bfB^r\in \MC(\calG^r)$ so that $P^{\bfB^1}$ is obtained from $P^{\bfB^r}$ by removing all the entries equal to $n+1$ in $P^{\bfB^r}$.
\end{Proposition}

\begin{proof}   
Let $\ell = \ell(M)$ and $\bfB^r\in \MC(\calG^r)$. Let $P=P^{\bfB^r}=(p_1,\ldots,p_{\ell-1})=(P_{1}, \ldots, P_{\beta})$ with $\beta=\si{\bfB^r}$.
 By \Cref{positions}, we know that exactly $r-1$ entries of $P^{\bfB^r}$ are equal to $n+1$ and thus there exist indices $i_1, \ldots, i_{r-1}$ such that  $p_{i_j}=n+1$ for all $j \in [r-1]$. 
 Let $P'=(p_1,\ldots,\widehat{p_{i_1}},\ldots, p_j,\ldots, \widehat{p_{i_{r-1}}}, \ldots,p_{\ell-1})$ be the tuple obtained by removing $p_{i_1},\ldots,p_{i_{r-1}}$ from $P^{\bfB^r}$. 

We claim that $P'=P^{\bfB^1}$ for some $\bfB^1\in \MC(\calG^1)$. Notice that for each $i_j$ with $p_{i_j}=n+1$ we have that $p_{i_j+s}< p_{i_j}$, where $s=\min\{k: p_{i_j+k}\in P'\}$. By construction  $p_{i_j+s}\in P'$ implies that $p_{i_j+s}$ was a permitted position change in $\bfB^r$, so the claim follows from \Cref{proving Strategy}.

 Conversely, let $\bfB^1 \in \MC(\calG^1)$ and $P^{\bfB^1}$ be its tuple of position changes. Write $P^{\bfB^1}=(P_1^1, \ldots, P_{\gamma}^{1})$ and $P_i^{1}=\{p_{i,1}< \cdots < p_{i, |P_i^1|}\}$ for each $i\in [\gamma]$, where $\gamma= \si{\bfB^1}$. We create a new tuple $P^r$ by appropriately augmenting $P^{\bfB^1}$ with additional $r-1$ entries, all equal to $n+1$.  
 
 For each $i\in [\min\{r-1, \gamma\}]$, let $P_i^r=\{p_{i,1}<\cdots <p_{i,|P_i^1|}< n+1\}$. If $\gamma=\min\{r-1, \gamma\}$, for each $i \in [\gamma+1, r-1]$ let $P_i^r=\{n+1\}$.  If instead $r-1= \min\{r-1, \gamma\}$, for all $i\in [r, \gamma]$ let $P_i^{r}=P_i^{1}$. 
 Now let $P^r=(P_1^r, \ldots, P_{k}^r)$, where $k=\max\{r-1, \gamma\}$. Then by \Cref{proving Strategy} it is clear that $P^r=P^{\bfB^r}$ for some $\bfB^r\in \MC(\calG^r)$. Thus, $P^{\bfB^1}$ is obtained from $P^{\bfB^r}$ by removing all the entries equal to $n+1$, as claimed.
 \end{proof}

\begin{Corollary}\label{CliquesG1}
     Adopt \Cref{LadderSetting}. For any maximal clique $\bfB^1\in \MC(\calG^1)$ there are exactly $\binom{\ell(L)+r-2}{r-1}$ maximal cliques $\bfB^r\in \MC(\calG^r)$ so that $P^{\bfB^1}$ can be obtained from $P^{\bfB^r}$ by removing all the $n+1$ entries of $P^{\bfB^r}$. 
\end{Corollary}

\begin{proof}
    Notice that in the proof of \Cref{addn+1} we exhibited one possible way to augment $P^{\bfB^1}$ in order to obtain $P^{\bfB^r}$. Recall that, for any $\bfC\in \MC(\calG^r)$, exactly $r-1$ entries of $P^{\bfC}$ equal to $n+1$; but by \Cref{proving Strategy}~\ref{jth i-1 after jth i}, there are no restrictions on which entries can equal $n+1$. Since $P^{\bfB^1}$ has $\ell(L)-1$ entries, there are exactly $\binom{\ell(L)+r-2}{r-1}$ ways to augment $P^{\bfB^1}$ with additional $r-1$ entries equal to $n+1$. 
    Thus, there exist exactly $\binom{\ell(L)+r-2}{r-1}$ maximal cliques $\bfB^r$ from which $\bfB^1$ is obtained by removing all entries equal to $n+1$.
\end{proof}

We are now able to relate the sequence indices of $\calA^r$ and $\calA^1$.

\begin{Corollary}\label{cor: si between 1 and r}
Adopt \Cref{LadderSetting} and let $\calA^{r}, \calA^1$ be the maximal cliques of $\calG^r$ and $\calG^1$ as in \Cref{MaxLengthClique},  respectively. Then, $\si{\calA^r}=\max\{r-1, \si{\calA^1}\}$.
\end{Corollary}

\begin{proof}
     We first notice that $r-1 \le \si{\calA^r}$ as $n+1$ appears with multiplicity $r-1$ in $P^{\calA^r}$. On the other hand, if we remove all the entries equal to $n+1$ from $P^{\calA^r}$, by \Cref{addn+1} and its proof we obtain a tuple $P'=P^{\bfB^1}$ for some $\bfB^1\in \MC(\calG^1)$ without increasing the sequence index. Hence $\si{\calA^r}\ge \si{\bfB^1}\ge \si{\calA^1}$, where the last inequality holds by \Cref{best reg ind}. Hence $\si{\calA^r}\ge\max\{r-1, \si{\calA^1}\}$.

Now notice that, as in the proof of \Cref{addn+1}, we can construct a tuple $P'$ from $P^{\calA^1}$ so that $P'=P^{\bfB^r}$ for some $\bfB^r\in \MC(\calG^r)$ and $\si{\bfB^r}=\max\{r-1, \si{\calA^1}\}$.
By \Cref{best reg ind} we then have $\si{\calA^r}\le \si{\bfB^r}=\max\{r-1, \si{\calA^1}\}$.
\end{proof}

We are finally ready to state the main theorem of this section, which establishes a formula for the regularity of $\calF(M)$ and explains how it relates to  the regularity of $\calF(L)$.

\begin{Theorem}\label{reg}
  Adopt \Cref{LadderSetting} and let $\calA^r, \calA^1$ be the maximal cliques of $\calG^r, \calG^1$ from \Cref{MaxLengthClique}.  Then 
  \begin{enumerate}[a]
      \item The regularity of $\calF(M)$ is 
        \begin{align*}
         \reg (\calF(M)) & = \ell(M)-1-\si{\calA^r}=\reg(\calF(L))+\min\{r-1, \si{\calA^1}\} \\
         &= r-1+\sum_{i=1}^{n}\Delta_i-\max\{r-1, \si{\calA^1}\}.
         \end{align*}
      \item The reduction number of $M$ is $\displaystyle{r(M) = \ell(M)-1-\si{\calA^r}=r(L)+\min\{r-1, \si{\calA^1}\}}$.
      \item The $a$-invariant of $\calF(M)$ is $ \displaystyle{ a(\mathcal{F}(M)) =-1-\si{\calA^r}=-1-\max\{r-1, \si{\calA^1}\}}$.
  \end{enumerate}
\end{Theorem}

\begin{proof}
Recall that $M=\bigoplus_{i=1}^{r}L$ and $N=\bigoplus_{i=1}^{r}(\ini_{\tau} L))$. As usual, let $\calJ$ denote the defining ideal of $\mathcal{F}(N)$. By \Cref{SAGBIandGB}~\ref{regToPd}, we know that $\reg(\calF(M))=\pd((\ini_{\sigma}(\calJ))^\vee)$. Moreover, since $(\ini_{\sigma}(\calJ))^{\vee}$ has linear quotients by \Cref{Thm: Linear Quotients}, it follows from \cite[Corollary 8.2.2]{HHMon} that $\pd((\ini_{\sigma}(\calJ))^{\vee})=\max\{\mu(Q_{\bfB}): \bfB\in \MC(\calG)\}$.
Hence, the first equality in (a) follows from \Cref{best reg ind} and \Cref{pdLowA}, and implies that $\reg (\calF(L)) =\ell(L)-1-\si{\calA^1}$.
Now \Cref{analyticSpread} and \Cref{cor: si between 1 and r} yield the remaining formulas for the regularity of $\calF(M)$. The calculations of the reduction number of $M$ and the $a$-invariant of $\calF(M)$ follow from \Cref{CosLS:2.18}.
\end{proof}

We emphasize that the formula for the regularity given by \Cref{reg} is explicit, as the sequence index $\si{\calA^r}$ can be easily calculated following the algorithm outlined in \Cref{MaxLengthClique}.

\begin{Example}
   We revisit \Cref{ex:Construction}. Recall that in this example $r=2$, $\si{\calA^2}=5$ and $\ell(M)=14$. Therefore, by \Cref{reg}, we have that 
   $$ r(M)=\reg(\calF(M))=\ell(M)-1-\si{\calA^2}=14-1-5=8  \quad \text{and} \quad a(\calF(M))=-1-\si{\calA^2}=-6.$$ 
\end{Example}

One can always use the algorithm defined in \Cref{MaxLengthClique} and \Cref{cor: si between 1 and r} to calculate $\si{\calA^r}$, but for matrices of large size this could be a lengthy process. However, having information on the $\epsilon_i$'s sometimes allows us to immediately compute $\si{\calA^r}$, and hence calculate $\reg(\calF(M))$ and $a(\calF(M))$. Recall that $\epsilon_i= u_{i+1}-u_i$ for $i \in [n-1]$; see \Cref{LadderSetting}.

\begin{Corollary}\label{regEpsilon_i}
    Adopt \Cref{LadderSetting} and let $\calA = \calA^r \in \MC(\calG^r)$ be the maximal clique defined in \Cref{MaxLengthClique}. Write $P^{\calA}=(P_1^{\calA}, \ldots, P_{\si{\calA}}^{\calA})$. 
    \begin{enumerate}[a]
        \item \label{all>2Delta} Suppose $\epsilon_{j}\ge 2$ for all $j\in [n-1]$. Then for all $i\in[n]$ we have   $i\in P^{\calA}_k$ if and only if $k\in [\Delta_i]$. In particular, in this case 
        \begin{enumerate}[i]
        \item $\si{\calA}=\max\limits_{i \in [n]}\{\Delta_i, r-1 \}$;
        \item $\reg(\calF(M)) = r-1+\sum_{i=1}^{n}\Delta_i-\max\limits_{i \in [n]}\{\Delta_i, r-1 \}$;
        \item $a(\calF(M))= -1 -\max\limits_{i \in [n]}\{\Delta_i, r-1 \}$.
        \end{enumerate}
        
        \item \label{All=1Delta} Suppose $\epsilon_j=1$ for all $j\in [n-1]$. Then for all $i\in[n]$, we have $i\in P^{\calA}_k$ if and only if $k\in [n-i+1,\Delta_i+n-i]$. In particular, in this case 
        \begin{enumerate}[i]
        \item $\si{\calA}=\max\limits_{i \in [n]}\{\Delta_i+n-i, r-1 \}$;
        \item $\reg(\calF(M))= r-1+\sum_{i=1}^{n}\Delta_i-\max\limits_{i \in [n]}\{\Delta_i+n-i, r-1 \}$;
        \item $a(\calF(M))=-1-\max\limits_{i \in [n]}\{\Delta_i+n-i, r-1 \}$.
        \end{enumerate}
        
    \end{enumerate}
\end{Corollary}

When $\bfX_{\bfS}$ is a generic matrix and $r=1$, the dimension, regularity and $a$-invariant of $\calF(I_{n}(\bfX_{\bfS}))$ are well--known. Using \Cref{regEpsilon_i}, we recover the known formulas for these invariants, and in particular a well--known result of Bruns and Herzog on $a(\calF(I_{n}(\bfX_{\bfS})))$, see \cite[Corollary~1.4]{BHAinvariant}.

\begin{Corollary} \label{recover BHanviarant} 
 Let $\bfX$ be a $n \times m$ generic matrix, $n \leq m,$ and let $I_n(\bfX)$ be the ideal of maximal minors of $\bfX$. Then
$\dim (\calF(I_n(\bfX)))=n(m-n)+1$, $\reg(\calF(I_n(\bfX)) )=(n-1)(m-n-1)$, and $a(\calF(I_n(\bfX)))=-m$.
\end{Corollary}

\begin{proof}
     Let $\bfX$ be an $n \times m$ generic matrix and let $I_n(\bfX)$ be the ideal of maximal minors of $\bfX$. 
     Let $S_1=[1,m-n+1]$, $S_2=[2,m-n+2], \ldots, S_n=[n,m]$, 
     $\bfS =\{S_1, \ldots, S_n\}$ and $L$ the ideal of maximal minors of $\bfX_{\bfS}$. Notice that $\ini_{\tau}(I_n(\bfX))=\ini_{\tau}(L)$. Then using \Cref{SAGBIandGB} we may replace $I_n(\bfX)$ by $L$ in our computations.
     
     Now, by the shape of $\bfX_{\bfS}$ we know that $\Delta_i=m-n$ for all $i\in[n]$ and $\epsilon_j=1$ for all $j\in[n-1]$. Hence by \Cref{analyticSpread} we have
     $$\dim (\calF(I_n(\bfX)))=\ell(I_n(\bfX))=\ell(L)=\sum_{i=1}^{n}\Delta_i+1=n(m-n)+1.$$
     Finally, the calculation for the regularity and $a$-invariant follow from \Cref{regEpsilon_i}.
\end{proof}

When $r=1$, formulas for the regularity, dimension and $a$-invariant of the special fiber ring in the case of a $2\times m$ sparse matrix $\bfX$ are given in \cite{CDFGLPS}. We next show that their results can be obtained as a consequence of \Cref{regEpsilon_i}. Without loss of generality, we may assume that, after permuting the rows and columns of $\bfX$, there exist nonnegative integers $\epsilon, s$ with $ \epsilon < m$ such that 
    $$\bfX= \begin{pmatrix}
x_{1,1} & \cdots & x_{1,\epsilon} & x_{1,\epsilon+1} & \cdots & x_{1,\epsilon+s} &  0 & \cdots & 0 \\
0 & \cdots & 0 & x_{2,\epsilon+1} & \cdots & x_{2,\epsilon+s} & x_{2,\epsilon+s+1} & \cdots & x_{2,m}
\\
\end{pmatrix};$$
see \cite[Section~2]{CDFGLPS}. Notice that, with the notation of \Cref{LadderSetting}, we have $\epsilon_1=\epsilon$, $\Delta_1=\epsilon+s-1$, and $\Delta_2=m-\epsilon-1$. Moreover, if $m=\epsilon+s$, then as in \Cref{shrinkmatrix} we can replace $x_{1, \epsilon+s}$ with a zero without altering the calculations of the invariants in the statement. Hence without loss of generality we may assume that $m-\epsilon-s>0$. Similarly, if $\epsilon=0$, then we can replace $x_{2,1}$ with a zero again by \Cref{shrinkmatrix}. Thus we may also assume $\epsilon>0$.

\begin{Corollary}\label{cor: sparse matrices}
    Let $\bfX$ be a $2\times m$ sparse matrix and $I_2(\bfX)$ be the ideal of maximal minors of $\bfX$. Then
    $\dim(\calF(I_2(\bfX)))=m+s-1$, 
    $\reg(\calF(I_2(\bfX)) )=\min\{m-3,\epsilon+s-1, m-\epsilon-1\}$, and $a(\calF(I_2(\bfX)))=\min\{-s-2,-m+\epsilon, -\epsilon-s\}$. 
\end{Corollary}

\begin{proof}
    The case when $\bfX$ is a generic matrix is already covered in \Cref{recover BHanviarant}. Hence, by the discussion above, we may assume that $\bfX$ has the form above. Moreover, $m-\epsilon-s>0, \epsilon_1=\epsilon>0$, $\Delta_1=\epsilon+s-1$, and $\Delta_2=m-\epsilon-1$. 

By \Cref{analyticSpread} it immediately follows that $\ell(I_2(\bfX))=\Delta_1+\Delta_2+1=m+s-1=\min\{m+s-1, 2m-\epsilon-2\}$, thus recovering \cite[Corollary~4.8~(a)]{CDFGLPS}.

By \Cref{regEpsilon_i}, when $\epsilon \ge 2$, $\reg(\calF(I_2(\bfX)) ) = \Delta_1+\Delta_2-\max\limits\{\Delta_1, \Delta
_2\}=\min\{\Delta_1,\Delta_2\}=\min\{\epsilon+s-1, m-\epsilon-1\}$. When $\epsilon =1$, $\reg(\calF(I_2(\bfX)) ) = \Delta_1+\Delta_2-\max\limits\{\Delta_1+1, \Delta_2\}=\min\{\Delta_1,\Delta_2-1\}=\min\{\epsilon+s-1, m-\epsilon-2\}=\min\{\epsilon+s-1, m-3\}$. Therefore, $\reg(\calF(I_2(\bfX)) )=\min\{m-3,\epsilon+s-1, m-\epsilon-1\}$, as shown also in \cite[Corollary~4.8~(c)]{CDFGLPS}.  Similarly, 
$a(\calF(I_2(\bfX)))=\min\{-s-2,-m+\epsilon, -\epsilon-s\}$, again recovering \cite[Corollary~4.8~(d)]{CDFGLPS}.
\end{proof}

Given an arbitrary ladder matrix $\bfX_{\bfS}$, the $\epsilon_i$'s need not satisfy the restrictions in \Cref{regEpsilon_i}. We thus devote the rest of this section to investigate ways to calculate $\si{\calA^1}$, and hence $\si{\calA^r}$, just using the data of the matrix $\bfX_{\bfS}$.
We introduce the following notation. 

\begin{Notation}\label{gorensteinRows}
 Adopt \Cref{LadderSetting} with $r=1$ and let $\calA^1 \in\MC(\calG^1)$ be the maximal clique defined in \Cref{MaxLengthClique}. 
Write $[n]=V\sqcup U$, where $$V=\{i \in [n-1]: \epsilon_i\ge 2\}\cup \{n\} \quad \text{and} \quad U=[n] \setminus V = \{i \in [n-1]: \epsilon_i =1\}.$$ 
Furthermore, there exist integers $\eta$ and $\eta'$ so that $V=\sqcup_{j=1}^{\eta}[g_j,h_j]$ and $U=\sqcup_{s=1}^{\eta'}[g'_s,h'_s]$, with $h_j+1< g_{j+1}$ for each $j \in [\eta-1]$, and $h'_s+1 <g'_{s+1}$ for each $s \in [\eta'-1]$.
\end{Notation}

Recall that in \Cref{regEpsilon_i} we could calculate $\si{\calA^1}$ explicitly because for each $i$ we were able to determine exactly for which $k$ we had $i\in P_k^{\calA^1}$. The following proposition similarly investigates for which $ i \in V$ or $i \in U$ we can determine $k$ so that $i\in P_k^{\calA^1}$. 

\begin{Proposition} \label{general facts for intervals}
    Adopt \Cref{LadderSetting} with $r=1$. Let  $\calA = \calA^1$ be the maximal clique defined in \Cref{MaxLengthClique} and let $P^{\calA}=(P_1^{\calA}, \ldots, P_{\si{\calA}}^{\calA})$ denote its tuple of position changes. Let $V$ and $U$ be as in \Cref{gorensteinRows}. The following hold:
\begin{enumerate}[a]
    \item \label{last V} If $i\in [g_{\eta}, h_{\eta}]$, then $i\in P_k^{\calA}$ for all $k\in[\Delta_i]$.

    \item \label{one interval for V} Suppose $i\in P_k^{\calA}$ if and only if  $k\in[a_i, b_i]$ with $a_i\le b_i$.  Then $i-1\in P_k^{\calA}$ if and only if $k\in  [\epsilon_{i-1}-1]\cup [a_i+1, \Delta_{i-1}+\max\{a_i-\epsilon_{i-1}+1, 0\}]$. 
    
    \item \label{blockObs=1pt} Let $i\in U$ and choose $s\in [\eta']$ so that $i\in [g_s', h_s']$. Suppose $h_s'+1 \in P_k^\calA$ for all $k\in [\Delta_{h_s'+1}]$. Then $i\in P_k^\calA$ for all $k\in [h_s'-i+2, \Delta_i + h_s'-i+1]$.       
\end{enumerate}    
\end{Proposition}

\begin{proof}
(a): Notice that $h_{\eta}=n$ and by \Cref{MaxLengthClique} we have  $n\in P_k^{\calA}$ for all $k\in [\Delta_n]$. If $g_{\eta}=n$, then there is nothing to show. Otherwise, $n-1\in [g_{\eta}, h_{\eta}]$ and since $\epsilon_{n-1}\ge 2$ in this case, then  $n-1\in P_k^{\calA}$ for all $k\in [\Delta_{n-1}]$, by \Cref{proving Strategy}~\ref{jth i-1 after jth i}. Inductively, we can then show that for all $i\in[g_{\eta}, h_{\eta}]$ we have  $i\in P_k^{\calA}$ for all $k\in [\Delta_i]$.

(b): Note that $\Big|[\epsilon_{i-1}-1]\cup [a_i+1, \Delta_{i-1}+\max\{a_i-\epsilon_{i-1}+1, 0\}]\Big|=\Delta_{i-1}$, where $[\epsilon_{i-1}-1]=\emptyset$ whenever $\epsilon_{i-1}=1$. Hence, by \Cref{PL} it suffices to show that $i-1\in P_k^{\calA}$ whenever $k\in [\epsilon_{i-1}-1]\cup [a_i+1, \Delta_{i-1}+\max\{a_i-\epsilon_{i-1}+1, 0\}]$. Since $i\in P_k^{\calA}$ for $k\in [a_i,b_i]$, by \Cref{proving Strategy} it then follows that $i-1\in [a_i+1, \Delta_{i-1}+\max\{a_i-\epsilon_{i-1}+1, 0\}]$. If $\epsilon_{i-1}=1$, then there is nothing else to show. 
If $\epsilon_{i-1}\ge 2$, then by \Cref{MaxLengthClique} we know $i-1\in P_k^{\calA}$ for all $k\in [\epsilon_{i-1}-1]$.

(c):  Let $i\in[g_{s}',h_{s}']$. Since $n\in V$, then $h_{s}'+1=g_{j}$ for some $j\in [\eta]$.   Since $\epsilon_{h_{s}'}=1$, then $h_{s}'\in P_k^{\calA}$ if and only if $k\in[2, \Delta_{h_{s}'}+1]$, by \Cref{proving Strategy} and the fact that $h_s'+1 \in P_k^\calA$ for all $k\in [\Delta_{h_s'+1}]$. Recursively, since $\epsilon_i=1$, we thus obtain that $i\in P_k^{\calA}$ if and only if $k\in[2+h_{s}'-i, \Delta_i+h_{s}'-i+1]$.
\end{proof}

\begin{Corollary} \label{BlackBox}
  Adopt \Cref{LadderSetting} with $r=1$. Let $\calA= \calA^1$ be the maximal clique defined in \Cref{MaxLengthClique} and write $P^{\calA}=(P_1^{\calA}, \ldots, P_{\si{\calA}}^{\calA})$ for its tuple of position changes.
  Let $V$ and $U$ be as in \Cref{gorensteinRows}. For each $j \in [\eta-1]$, write $h_j+1=g_{s}'$ for some $s\in [\eta']$ and suppose that $h_s'+1\in P^{\calA}_k$ for all $k\in [\Delta_{h_s'+1}]$. Then $h_j\in P_k^{\calA}$ for all $k\in [\Delta_{h_j}]$ if and only if $\Delta_{h_j}=\epsilon_{h_j}-1$ or $\epsilon_{h_j}\ge h_s'-g_s'+3$. 
\end{Corollary}

\begin{proof}
   Since $h_s'+1\in P^{\calA}_k$ for all $k\in [\Delta_{h_s'+1}]$, then by \Cref{general facts for intervals}~\ref{blockObs=1pt} we have $g_s'\in P_k^{\calA}$ for all $k\in[h_s'-g_s'+2, \Delta_{g_s'}+h_s'-g_s'+1]$.  Hence, by \Cref{general facts for intervals}~\ref{one interval for V} it follows that  $h_j\in P_k^{\calA}$ if and only if $k\in [\epsilon_{h_j}-1]\cup[h_s'-g_s'+3, \Delta_{h_j}+\max\{h_s'-g_s'+3-\epsilon_{h_j}, 0\}]$. It remains to show that $[\epsilon_{h_j}-1]\cup[h_s'-g_s'+3, \Delta_{h_j}+\max\{h_s'-g_s'+3-\epsilon_{h_j}, 0\}]=[\Delta_{h_j}]$ if and only if $\Delta_{h_j}=\epsilon_{h_j}-1$ or $\epsilon_{h_j}\ge h_s'-g_s'+3$. 

Suppose $[\epsilon_{h_j}-1]\cup[h_s'-g_s'+3, \Delta_{h_j}+\max\{h_s'-g_s'+3-\epsilon_{h_j}, 0\}]=[\Delta_{h_j}]$.  If $\Delta_{h_j}\neq\epsilon_{h_j}-1$, then $\Delta_{h_j}>\epsilon_{h_j}-1$, see \Cref{LadderSetting}. Hence, $\Delta_{h_j}+\max\{h_s'-g_s'+3-\epsilon_{h_j}, 0\}=\Delta_{h_j}$ and thus $\max\{h_s'-g_s'+3-\epsilon_{h_j}, 0\}=0$. In other words, $\epsilon_{h_j}\ge h_s'-g_s'+3$.

Conversely, if $\epsilon_{h_j}\ge h_s'-g_s'+3$, then $[h_s'-g_s'+3, \Delta_{h_j}+\max\{h_s'-g_s'+3-\epsilon_{h_j}, 0\}]=[h_s'-g_s'+3, \Delta_{h_j}]$. Hence 
$[\epsilon_{h_j}-1]\cup[h_s'-g_s'+3, \Delta_{h_j}]=[\Delta_{h_j}]$, since $\Delta_{h_j}\ge \epsilon_{h_j}-1$ and $\epsilon_{h_j}\ge h_s'-g_s'+3$. If instead, $\Delta_{h_j}=\epsilon_{h_j}-1$ and $\epsilon_{h_j}< h_s'-g_s'+3$, then $[h_s'-g_s'+3, \Delta_{h_j}+\max\{h_s'-g_s'+3-\epsilon_{h_j}, 0\}]=[h_s'-g_s'+3, h_s'-g_s'+2]=\emptyset$. Therefore, $[\epsilon_{h_j}-1]\cup[h_s'-g_s'+3, \Delta_{h_j}+\max\{h_s'-g_s'+3-\epsilon_{h_j}, 0\}]=[\Delta_{h_j}]$, as claimed. 
\end{proof}

Using the assumptions from \Cref{BlackBox}, we are now able to calculate $\si{\calA}$ explicitly. 

\begin{Theorem}\label{thm: si extended formula}
    Adopt \Cref{LadderSetting} with $r=1$ and let $V=\sqcup_{j=1}^{\eta}[g_j,h_j]$ and $U=\sqcup_{s=1}^{\eta'}[g'_s,h'_s]$ be as in \Cref{gorensteinRows}. Let $\calA=\calA^1$ be the maximal clique defined in \Cref{MaxLengthClique}.  For each $j \in [\eta-1]$, choose $s\in [\eta']$ such that 
  $h_j+1=g_{s}'$ and suppose that either $ h'_s-g'_s+3 \le \epsilon_{h_j}$ or $\Delta_{h_j}= \epsilon_{h_j}-1$. Then 
    $$\si{\calA}=\max\{\Delta_i, \Delta_t+h_s'-t+1 \mid i\in V, t\in U \text{ so that } t\in [g_s',h_s'], s\in [\eta']\}.$$
\end{Theorem}

\begin{proof}
    We claim that the following statements hold.
        \begin{enumerate}[1]
\item  For all $i\in V$ we have $i\in P_k^{\calA}$ if and only if  $k\in [\Delta_i]$.
\item For all $t\in U$ we have $t\in P_k^{\calA}$ if and only if $k\in [h_s'-t+2, \Delta_t + h_s'-t+1]$.
\end{enumerate}

We prove (1) by reverse induction on the intervals of $V$. First, notice that for any $i\in[g_{\eta},h_{\eta}]$ we have $i\in P_k^{\calA}$ if and only if $k\in [\Delta_{i}]$, by \Cref{general facts for intervals}~\ref{last V}.
Now, let $j\in [\eta-1]$ and suppose by induction that for any $i\in [g_{j+1}, h_{j+1}]$ we have $i\in P_k^{\calA}$ if and only if $k\in [\Delta_{i}]$.
Since $j<\eta$, then $h_j+1=g_s'$ for some $s\in[\eta']$ and either $\epsilon_{h_j}\ge h_s'-g_s'+3$ or $\Delta_{h_j}=\epsilon_{h_j}-1$ by our assumptions. Then, by \Cref{BlackBox} it follows that $h_{j}\in P_k^{\calA}$ if and only if $k\in[\Delta_{h_{j}}]$. 
Now, let $i\in [g_j,h_j]$ and notice that since $i\in V$, then   $i\in P_k^{\calA}$ if and only if  $k\in [\Delta_i]$,  by \Cref{proving Strategy} and the fact that $h_j\in P_k^{\calA}$ for all $k\in [\Delta_{h_j}]$ .

To prove (2), let $t\in U$ and $s\in [\eta']$ such that $t \in[g_s', h_s']$. By the definition of $V$ we have $h_s'+1\in V$ and  hence $h_s'+1\in P_k^\calA$ for all $k\in [\Delta_{h_s'+1}]$ by (1). Thus $t\in P_k^{\calA}$ if and only if $k\in [h_s'-t+2, \Delta_t + h_s'-t+1]$ by \Cref{general facts for intervals}~\ref{blockObs=1pt}. 

Finally, the formula for $\si{\calA}$ follows from (1) and (2).
\end{proof}

When the assumptions of \Cref{thm: si extended formula} are satisfied, one can calculate $\si{\calA^1}$ (and hence $\si{\calA^r}$) from the data of the ladder matrix $\bfX_{\bfS}$. In particular, if the size of $\bfX_{\bfS}$ is large, this method is more efficient than the algorithm outlined in \Cref{MaxLengthClique}.

\begin{Example}
    Let $n=6$, $m=18$,  $\bfS= \{ [1,7], [7,12], [8,13], [9,14], [12, 17], [14,18] \}$, and let $\bfX_{\bfS}$ be the ladder matrix associated to $\bfS$. Let $L=I_6(\bfX_{\bfS})$. 

    Notice that $\Delta_1=6$, $\Delta_2=\Delta_3=\Delta_4=\Delta_5=5$, and $\Delta_6=4$, so $\ell(L)=1+\sum_{i=1}^{6}\Delta_i=31$. 
    Moreover, $\epsilon_1=6$, $\epsilon_2=\epsilon_3=1$,  $\epsilon_4=3$, and $\epsilon_5=2$. 
    Thus, $V=[1]\sqcup [4,6]$ and $U=[2,3]=[g_1',h_1']$ as in \Cref{gorensteinRows}. Now notice that $h_1'-g_1'+3=4<\epsilon_1$ and the conditions of \Cref{thm: si extended formula} are satisfied. Hence,
    $$\si{\calA^1}=\max\{\Delta_1, \Delta_4, \Delta_5, \Delta_6, \Delta_2+3-2+1, \Delta_3+3-3+1\}=\max\{6,5,5,4,5+2, 5+1\}=7.$$ 
    Therefore, by \Cref{cor: si between 1 and r} we have $\si{\calA^r}=\max\{r-1, \si{\calA^1}\}=\max\{r-1, 7\}$. Hence, by \Cref{reg} we have $\reg(\calF(L))=\ell(L)-1-\si{\calA^1
    }=31-1-7=23\,$ and 
    $$ \reg(\calF(M))=r+29-\max\{r-1, 7\}= \begin{cases}
        r+22 & \text{ if }\,  r\le 8, \\
        30 & \text{ if }\, r>8 .
    \end{cases}$$
\end{Example}

    Recall by \Cref{SAGBIandGB} that $\ell(M)=\ell(N)$ and $\reg(\calF(M))=\reg(\calF(N))$; moreover, $\calF(N)$ is a Hibi ring, see \Cref{FiberLadder}. With the Hibi ring structure in mind, formulas for $\ell(N)$ and $\reg(\calF(N))$ can be given in terms of combinatorial data from the poset $\calP$ of the join-irreducible elements of the lattice $\calL \times [r]$; see \cite[Theorems~6.38~and~6.42]{HerHiOh}. 
    Our formulas for $\ell(N)$ and $\reg(\calF(N))$ from \Cref{analyticSpread} and \Cref{reg} prove that these invariants can be calculated using data from the ladder matrix $\bfX_{\bfS}$ only, without any knowledge of $\calP$; see also \Cref{thm: si extended formula}. 
    As a corollary, we are able to interpret the cardinality and rank of $\calP$ in terms of data of the  ladder matrix $\bfX_{\bfS}$.

\begin{Corollary}\label{cor: regposet}
   Adopt \Cref{LadderSetting} and let $\calP$ be the poset of the join-irreducible elements of the lattice $\calL \times [r]$. Then, $|\calP| = r-1+\sum _{i=1}^{n}\Delta_i$ and $\rank(\calP)= \si{\calA^r} -1$.
\end{Corollary}

\begin{proof}
    By \cite[Theorem 6.38]{HerHiOh}, the dimension of $\calF(N)$ is $\ell(N)= |\calP| +1$. It then follows from \Cref{SAGBIandGB}~\ref{analytic} and \Cref{analyticSpread} that $|\calP| = r-1+\sum _{i=1}^{n}\Delta_i $. 
    On the other hand, by \Cref{SAGBIandGB}~\ref{regToPd} and \cite[Theorem 6.42]{HerHiOh}, $\reg(\calF(M)) = \reg(\calF(N)) = |\calP| -\rank(\calP)-1$, whence \Cref{reg} implies that $\rank(\calP)= \si{\calA^r} -1$, as claimed. 
\end{proof}

\section{Multiplicity}\label{sec:multi}

The objective of this section is to obtain a combinatorial formula for the multiplicity of $\calF(M)$, where 
$M=\bigoplus_{i=1}^rL$ with $L=I_n(\bfX_{\bfS})$ and $\bfX_{\bfS}$ are as in \Cref{LadderSetting}. For $n=2$ and $r=1$, the multiplicity of $\calF(L)$ was determined in \cite[Corollary 4.9]{CDFGLPS};
see also \cite[Corollary~4.5]{CHV96} for the calculation of $e(\calF(L))$ when $L$ is the ideal of maximal minors of a generic $2 \times m$ matrix.

In light of \Cref{SAGBIandGB}~\ref{multiplicity eq}, the task is reduced to calculating the multiplicity of $\calF(N)$. Since the latter is a Hibi ring (see \Cref{FiberLadder}), it then follows from \cite[Theorem 3.9]{BLHibi} that $e(\calF(N))$ is the number of maximal chains in  $\calL \times [r]$. In the next proposition, we reinterpret this result in terms of maximal cliques of the graph $\calG^r$ introduced in \Cref{LadderSetting}.

\begin{Proposition}
    \label{prop:mult_is_number_of_max_clique}
    Adopt \Cref{LadderSetting}. Then $e(\calF(M))=|\MC(\calG^r)|$.
\end{Proposition}

\begin{proof}
By \Cref{SAGBIandGB}~\ref{multiplicity eq}, we have $e(\calF(M))= e(T/\ini_{\sigma}(\calJ))$. 
By \cite[Lemma~4.1]{Terai},  this number is the minimal number of generators of the equigenerated squarefree monomial ideal $(\ini_{\sigma}(\calJ))^\vee$, which coincides with the number of maximal cliques of $\calG^r$ by \Cref{rmk:maximal_cliques}.
\end{proof}

We next prove that there is a natural formula relating the multiplicity of $\calF(M)$ with that of $\calF(L)$, in view of the  correspondence between the maximal cliques of $\calG^r$ and the maximal cliques of $\calG^1$ discussed in \Cref{addn+1}. 

\begin{Proposition}\label{multi-tor}
   Adopt \Cref{LadderSetting}. Then 
    $$e(\calF(M))=e(\calF(L))\binom{\ell(L)+r-2}{r-1} = e(\calF(L))\binom{\ell(M)-1}{r-1},$$
    where $\ell(L)$ and $\ell(M)$ are the analytic spreads of $L$ and $M$, respectively.
\end{Proposition}

\begin{proof}
    By \Cref{SAGBIandGB}~\ref{multiplicity eq} and \Cref{analyticSpread}, it suffices to prove that 
    $$e(\calF(N)) = e(\calF(\ini_{\tau} (L)))\binom{\ell(L)+r-2}{r-1}.$$ 
    Let $\calG^r$ and $\calG^1$ be the graphs as in \Cref{LadderSetting}. The claim follows once we show that $|\MC(\calG^r)| = \binom{\ell(L)+r-2}{r-1} |\MC(\calG^1)|$, which is clear, using \Cref{addn+1}, \Cref{CliquesG1}, and their proofs.
\end{proof}

With \Cref{multi-tor} in mind, we now focus on the case $r=1$ and work on determining the multiplicity of $\calF(L)$. By \Cref{prop:mult_is_number_of_max_clique} it suffices to compute the cardinality of the set of maximal cliques of $\calG^1$. When $\bfX$ is a generic $n \times m $ matrix, then $\calF(I_n(\bfX))$ is the coordinate ring of a Grassmannian variety and its multiplicity $e(\calF(I_n(\bfX)))$ can then be calculated by counting the number of standard Young tableaux associated to a certain partition, see for example, \cite[Theorem 2.31]{Grassmannians} or \cite[Chapter 9]{FultonYoung}.
Inspired by this result, we next proceed to show that, when $\bfX_{\bfS}$ is an arbitrary ladder matrix, one instead needs to count the number of standard skew Young tableaux associated to a particular skew partition. 
We begin by recalling the notion of a standard skew Young tableau.

\begin{Definition}
Let $\Lambda$ be a positive integer. A \emph{partition} of $\Lambda$ is a sequence $\bflam=(\lambda_1,\ldots,\lambda_k)$ of positive integers such that $\lambda_i\ge \lambda_{i+1}$ for all $i\in [k-1]$ and $\sum_{i =1}^k \lambda_i = \Lambda$. A \emph{Young diagram} (or \emph{Ferrers diagram}), $F^{\bflam}$, of shape  $\bflam$ is a collection of empty cells having $k$ left-justified rows, with row $i$ containing $\lambda_i$ cells for $i \in [k]$. We assume the convention that the cell $(i,j)$ is located in row $i$ and column $j$.

Let $\bflam=(\lambda_1,\ldots,\lambda_k)$ be a partition and let $\bfmu=(\mu_1,\ldots,\mu_k)$ be a sequence of nonnegative integers such that $\mu_i\ge \mu_{i+1}$ for all $i\in [k-1]$. If $\lambda_i\ge \mu_i$ for all $i \in [k]$, we say that $\bflam \ge \bfmu$. Given such $\bflam, \bfmu$ with $\bflam \ge \bfmu$, one defines the \emph{skew Young diagram} $F^{\bflam/\bfmu}$ by superimposing the Young diagram $F^{\bflam}$ with $F^{\bfmu}$ and removing all the cells of  $F^{\bfmu}$, with the convention that $F^{\bfmu}$ has no cells in row $i$ if $\mu_i=0$. In this case, we call $\bflam/\bfmu$ a \emph{skew partition} of $K=\sum_{i=1}^k(\lambda_i-\mu_i)$.

 A \emph{skew Young tableau} of a skew partition $\bflam/\bfmu$ of $K$ is a skew Young diagram of shape $\bflam/\bfmu$, whose cells are filled with the integers $1, \ldots, K$ without any repetition. By abuse of notation, we denote such a skew Young tableau by $F^{\bflam/\bfmu}$.  A skew Young tableau $F^{\bflam/\bfmu}$ is said to be \emph{standard} if the integers $1,2,\ldots,K$ are inserted in such a way that the entries in each row and in each column are strictly increasing, when read left to right and top to bottom, respectively.
\end{Definition}

\begin{Example}\label{ex: skew young}
    Let $\bflam=(6,4,4)$ and $\bfmu=(1,1,0)$. Then the skew Young diagram of shape $\bflam/\bfmu$ is shown in \Cref{fig:skew example} and a  standard skew Young tableau is shown in \Cref{fig:stand skew example}. 
\begin{figure}[H]    
$\ydiagram{6,4,4}*[*(gray)]{1,1}$  
\caption{The skew Young diagram of shape $\bflam/\bfmu$}\label{fig:skew example}
\end{figure} 
\begin{figure}[H]    
$\begin{ytableau}
*(gray) &2& 4& 7& 10 &12\\
*(gray) & 3& 6 &9 \\
 1& 5& 8 &11 
\end{ytableau}$ 
\caption{A standard skew Young tableau of shape $\bflam/\bfmu$}\label{fig:stand skew example}
\end{figure}

\end{Example}

Having established the relevant terminology, we now return to our goal of calculating $e(\calF(L))$. 

\begin{Theorem}\label{thm: clique skew tableau}
     Adopt \Cref{LadderSetting} with $r=1$. For all $i\in [n-1]$ let $\mu_i=\sum_{j=1}^{n-i}(\epsilon_{j}-1)$ and let $\mu_n=0$. Let $\bfmu=(\mu_{1}, \ldots, \mu_{n})$ and $\bflam=(\mu_1+\Delta_n, \ldots, \mu_{n-1}+\Delta_2, \mu_n+\Delta_1)$. 
    
    Then there is a one-to-one correspondence between the set of the maximal cliques of $\mathcal{G}^1$ and the set of standard skew Young tableaux of shape $\bflam/\bfmu$.
\end{Theorem}

\begin{proof} 
   Recall that by \Cref{LadderSetting} we have for all $i\in [n-1]$,  $\Delta_{i}\le \Delta_{i+1}+\epsilon_{i}-1$  and hence  $\lambda_i=\mu_i+\Delta_{n+1-i}\ge \mu_{i+1}+ \Delta_{n-i}=\lambda_{i+1}$. Therefore, $\bflam/\bfmu$ is a skew partition and by \Cref{analyticSpread} we have $\sum_{i=1}^{n}(\lambda_i-\mu_i)=\sum_{i=1}^n{\Delta_i}=\ell-1$, with $\ell=\ell(L)$.  Let $\calG=\calG^1$. We will create a one-to-one correspondence between the set $\MC(\calG)$ and the set $\mathcal{Y}$ of all standard skew Young tableaux of shape $\bflam/\bfmu$.

Let $\bfB\in \MC(\calG)$ and let $P^{\bfB}=(p_1,\ldots,p_{\ell-1})$ be its tuple of position changes. For each $i\in [n]$, suppose $p_{i_1}=\cdots =p_{i_{\Delta_i}}=i \in P^{\bfB}$. Without loss of generality we may assume that $i_1<i_2<\cdots <i_{\Delta_i}$. Let $F^{\bflam/\bfmu}$ be the skew Young diagram of shape $\bflam/\bfmu$ and for each $i$ assign the numbers $i_1, \ldots, i_{\Delta_i}$ from left to right in the $(n-(i-1))$-th row of $F^{\bflam/\bfmu}$. We claim that with this assignment $F^{\bflam/\bfmu}$ becomes a standard skew Young tableaux.

Since the entries in each row of $F^{\bflam/\bfmu}$ are increasing from left to right, we only need to show that the columns have strictly increasing entries from top to bottom. For each $i $ and $j$, $i_j$ appears in the $(n-i+1, \mu_{n-i+1}+j)$-th cell of $F^{\bflam}$. Let $k$ be the entry in the $(n-i+2, \mu_{n-i+1}+j)$-th cell of  $F^{\bflam}$ and thus $p_k=i-1$. Now by construction $\mu_{n-i+1}+j=\mu_{n-i+2}+\epsilon_{i-1}-1+j$ and therefore $k$ is the $(j+\epsilon_{i-1}-1)$-th copy of $i-1$ in $P^{\bfB}$. Hence, by \Cref{proving Strategy}~\ref{jth i-1 after jth i} we have that $i_j<k$, as claimed.

   Conversely, let  $F_1^{\bflam/\bfmu}$ be a standard skew Young tableau of shape $\bflam/\bfmu$ and let $F_{1}^{\bflam}$ be the corresponding Young diagram of shape $\bflam$. Notice that there are exactly $\ell-1=\sum_{i=1}^{n}\Delta_i$ cells in $F_1^{\bflam/\bfmu}$ and each number $k\in[\ell-1]$ will appear exactly once. For each $k\in [\ell-1]$, set $q_k=i$ whenever the number $k$ appears in the $(n-i+1)$-th row of $F_1^{\bflam/\bfmu}$. Let $P=(q_1, \ldots, q_{\ell-1})$. We claim that $P=P^{\bfB_1}$ for some $\bfB_1\in \MC(\calG)$. 
   
   Since the $(n-i+1)$-th row of $F_1^{\bflam/\bfmu}$ has $\Delta_i$ cells, then the value of $q_k$ must appear exactly $ \Delta_{i}$ times. Thus, the collection $\{q_1, \ldots, q_{\ell-1}\}$ is $P_{\calL \times[1]}$. Write $\bfB_1=[\bdb^1, \ldots, \bdb^{\ell}]$, where $\bdb^1=\bdu$ and $\bdb^{i+1}=\bdb^{i}+\bde_{q_i}$ for $i \in [\ell-1]$. We will prove that $\bfB_1 \in \MC(\calG)$ via \Cref{proving Strategy}. 
  
   Let $i\in [2,n]$ and let $q_{i_1}=\cdots =q_{i_{\Delta_i}}=i \in P$ with $i_1<i_2<\cdots <i_{\Delta_i}$. Let $q_{i_j}=i$ be the $j$-th entry of $i$, that is $i_j$ appears in the $(n-i+1, \mu_{n-i+1}+j)$-th cell  of $F_1^{\bflam}$, the corresponding standard Young tableau of shape $\bflam$. 
  Now consider the $(n-i+2, \mu_{n-i+1}+j)$-th cell of $F_1^{\bflam}$ and suppose the entry in that cell is $j'$. Then $q_{j'}=i-1$. Moreover, since $\mu_{n-i+1}+j=\mu_{n-i+2}+\epsilon_{i-1}-1+j$, then $q_{j'}$ is the $(\epsilon_{i-1}-1+j)$-th entry of $i-1$ in $P$.
   Since in a standard skew Young tableau each column is strictly increasing, then $i_j<j'$; in other words, in $P$ the $j$-th entry of $i$, which is $q_{i_j}$, appears before $q_{j'}$, which is the $(\epsilon_{i-1}-1+j)$-th entry of $i-1$. Therefore, $\bfB_1$ is a maximal clique and $P=P^{\bfB_1}$ by \Cref{proving Strategy}.
   \end{proof}

\begin{Example} 
We revisit the standard skew Young tableau from \Cref{ex: skew young}.

\begin{figure}[H]    
$\begin{ytableau}
*(gray) &2& 4& 7& 10 &12\\
*(gray) & 3& 6 &9 \\
 1& 5& 8 &11 
\end{ytableau}$ 
\end{figure}

Using the algorithm described in the proof of \Cref{thm: clique skew tableau}, we determine the tuple of position changes $(q_1, \ldots, q_{\ell-1})$ of the maximal clique $\bfB$ corresponding to this standard skew Young tableau.
The second row is the $(3-2+1)$-th row and thus $q_3=q_6=q_9=2$.  Similarly, from the first and third rows we obtain $q_2=q_4=q_7=q_{10}=q_{12}=3$ and $q_1=q_5=q_8=q_{11}=1$, respectively. Hence $P=(1,3,2,3,1,2,3,1,2,3,1,3)$ and $P=P^{\bfB}$ for some $\bfB\in \MC(\calG)$. In fact, with similar calculations as in \Cref{ex:Construction}, we can see that $\bfB=\calA^{1}$ for the graph $\calG^1$ corresponding to the ladder matrix from \Cref{LadderMatrix}.
\end{Example}

In view of \Cref{thm: clique skew tableau}, the task of counting the number of maximal cliques of $\mathcal{G}^1$ is reduced to counting the number of standard skew Young tableaux of shape $\bflam/\bfmu$. To this end, we will use the notions of hook length and excited diagrams, which we now recall.

\begin{Definition}
Let $F^{\bflam}$ be a Young diagram of a partition $\bflam$. 
\begin{enumerate}[a]
    \item The \emph{hook length} of a cell $c=(i,j)$, denoted by $h_{\bflam}(c)$, is the total number of cells $(i',j')$ with $i'=i, \,j'\ge j$ or $i'>i, \, j'=j$. 
    In other words, the hook length $h_{\bflam}(i,j)$ counts the total number of cells on the right and below the cell $(i,j)$, with $(i,j)$ counted once.
    \item Let $D$ be a subset of $F^{\bflam}$. A cell $c=(i,j)\in D$ is called \emph{active} if $(i+1, j)$, $(i, j+1)$ and $(i+1, j+1)$ are all in $F^{\bflam}\setminus D$. For an active cell $c=(i,j) \in D$, define $\alpha_{c}(D)$ to be the set obtained by replacing $(i,j)\in D$ by $(i+1,j+1)$. We call this replacement an \emph{excited move}.
    \item Let $\bflam/\bfmu$ be a skew partition. An \emph{excited diagram} of $\bflam/\bfmu$  is a subset of $F^{\bflam}$ of size $\sum \mu_i$, obtained from $F^{\bfmu}$ by a sequence of excited moves in $F^{\bflam}$. 
    Let $\calE(\bflam/\bfmu)$ denote the set of all excited diagrams of  $F^{\bflam/\bfmu}$.
\end{enumerate}
\end{Definition}

We illustrate the notion of excited diagram in the following example.  For more information on excited diagrams, see for example \cite{MORALES2018350}.
\begin{Example}
    Let $\bflam=(4,3,3)$ and $\bfmu=(2,1,0)$. The excited diagrams of $F^{\bflam/\bfmu}$ are shown in \Cref{Excited diagrams} below. Following the arrows, each diagram is obtained from the previous one by performing exactly one excited move.
    
\begin{figure}[H]
\begin{tikzpicture}
  \node at (0,0) {$\begin{ytableau}
*(yellow) &*(yellow) &  &\\
*(yellow) & &   \\
 & &   
\end{ytableau}$};
\node at (3.5,0) {$\begin{ytableau}
*(yellow) &*(yellow) &  &\\
 & &   \\
 &*(yellow) & 
\end{ytableau}$};
\node at (7,0) {$\begin{ytableau}
*(yellow) & &  &\\
&&*(yellow)   \\
 & *(yellow)&
\end{ytableau}$};
\node at (10.5,0) {$\begin{ytableau}
 *(white) & &  &\\
&*(yellow)&*(yellow)   \\
 & *(yellow)&
\end{ytableau}$};

\node at (3.5,-2.75) { $\begin{ytableau}
*(yellow) & &  &\\
*(yellow) & & *(yellow)  \\
 & &
\end{ytableau}$};

\draw[->, thick] (1.5,0) -- (2,0) node[midway, above] {};
\draw[->, thick] (5,0) -- (5.5,0) node[midway, above] {};
\draw[->, thick] (8.5,0) -- (9,0) node[midway, above] {};
\draw[->, thick] (1.2,-1.6) -- (1.75,-1.9) node[midway, above] {};
\draw[->, thick] (5.2,-1.8) -- (5.75,-1.5) node[midway, above] {};
  \end{tikzpicture}
  \caption{The excited diagrams $\calE(\bflam/\bfmu)$ for the skew partition $\bflam/\bfmu$} \label{Excited diagrams}
\end{figure}

\end{Example}

The following well--known result calculates the number of standard skew Young tableaux of a given shape.

\begin{Theorem}[{\cite[Theorem 1.2]{MORALES2018350}}] 
\label{hookthm}
Let $\bflam/\bfmu$ be a skew partition and let $K=\sum_{i=1}^k(\lambda_i-\mu_i)$. Then the number $f^{\bflam/\bfmu}$ of standard skew Young tableaux of shape $\bflam/\bfmu$ is given by 
$$ f^{\bflam/\bfmu}= K! \mathlarger{\sum}_{D \in \calE(\bflam/\bfmu)} \frac{1}{\prod_{c \in F^{\bflam}\setminus D} h_{\bflam}(c)}.$$
\end{Theorem}

\begin{Example}\label{ex: excited}
Let $\bflam=(6,4,4)$ and $\bfmu=(1,1,0)$ be as in \Cref{ex: skew young}. We use the formula in \Cref{hookthm} to compute the number of standard skew Young tableaux of shape $\bflam/\bfmu$. We need to consider all possible excited diagrams associated to $\bflam/\bfmu$, which we show in \Cref{fig:skew example excited} below, along with the hook lengths for each cell of $F^{\bflam}$.
 
\begin{figure}[H]    
$\begin{ytableau}
*(yellow) 8 &7& 6& 5&2  &1\\
*(yellow) 5 & 4& 3 &2 \\
 4& 3& 2 &1 
\end{ytableau}$ \quad $\begin{ytableau}
*(yellow)8 &7& 6& 5&2  &1\\
 5& 4& 3 &2 \\
 4& *(yellow) 3& 2 &1 
\end{ytableau}$ \quad $\begin{ytableau}
8 &7& 6& 5&2  &1\\
 5& *(yellow) 4 & 3 &2 \\
 4& *(yellow) 3 & 2 &1 
\end{ytableau}$
\caption{The excited diagrams $\calE(\bflam/\bfmu)$ for the skew partition $\bflam/\bfmu$, with the hook lengths $h_{\bflam}(c)$.} \label{fig:skew example excited}
\end{figure}

Since $K=\sum_{i=1}^3(\lambda_i-\mu_i)=12$, then 
 $$f^{\bflam/\bfmu}=12!\Big[\frac{1}{7!\cdot 4! \cdot 2}+\frac{1}{7!\cdot 5 \cdot 4 \cdot 2 \cdot 2}+\frac{1}{8!\cdot 5 \cdot 2 \cdot 2}\Big]=3762.$$
\end{Example}

In our final result, we are now able to give an explicit formula for the multiplicity of both $\calF(L)$ and $\calF(M)$.

\begin{Theorem}\label{thm:multiplicity}
   Adopt \Cref{LadderSetting}. For all $i\in [n-1]$ let $\mu_i=\sum_{j=1}^{n-i}(\epsilon_{j}-1)$ and let $\mu_n=0$. Let $\bfmu=(\mu_{1}, \ldots, \mu_{n})$ and $\bflam=(\mu_1+\Delta_n, \ldots, \mu_{n-1}+\Delta_2, \mu_n+\Delta_1)$.  
Then 
    $$e(\calF(L))=(\ell(L)-1)!\mathlarger{\sum}_{D \in \calE(\bflam/\bfmu)} \frac{1}{\prod_{c \in F^{\bflam}\setminus D} h_{\bflam}(c)} = \left(\sum_{i=1}^n \Delta_i \right)! \mathlarger{\sum}_{D \in \calE(\bflam/\bfmu)} \frac{1}{\prod_{c \in F^{\bflam}\setminus D} h_{\bflam}(c)}$$
and
$$ e(\calF(M))=\binom{\ell(L)+r-2}{r-1}(\ell(L)-1)!\mathlarger{\sum}_{D \in \calE(\bflam/\bfmu)} \frac{1}{\prod_{c \in F^{\bflam}\setminus D} h_{\bflam}(c)},$$    
    where $h_{\bflam}(c)$ denotes the hook length of the cell $c$ of the Young diagram $F^{\bflam}$ and $\calE(\bflam/\bfmu)$ is the set of excited diagrams of $\bflam/\bfmu$.
 \end{Theorem}
 
\begin{proof}
The formula for $e(\calF(L))$ follows from \Cref{prop:mult_is_number_of_max_clique}, \Cref{thm: clique skew tableau}, and \Cref{hookthm}. Finally, \Cref{multi-tor} yields the formula for $e(\calF(M))$.
\end{proof}

\begin{Example}
     We revisit one last time the ladder matrix of \Cref{LadderMatrix} and we compute the multiplicity of $\calF(M)$ when $r=2$. Recall that in this case, $\Delta_1=4, \Delta_2=3, \Delta_3=5$ and $\epsilon_1=2$ and $\epsilon_2=1$. Therefore, $\mu_1=1, \mu_2=1$, and $\mu_3=0$. Hence $\bflam=(6,4,4)$ and $\bfmu=(1,1,0)$. Notice that in \Cref{ex: excited} we already counted the number of skew Young tableaux of shape $\bflam/\bfmu$. Therefore, by \Cref{thm:multiplicity} we have $e(\calF(L))=3762$ and $e(\calF(M))=13(3762)=48906$, since $\ell(L)= 13$. 
\end{Example}

\vspace{0.5cm}
\begin{acknowledgment*}
 This project originated at the Women in Commutative Algebra II Workshop held in Trento, Italy. We gratefully acknowledge the support of the Women in Commutative Algebra Network and CIRM (Trento). The travel of several participants to Trento was partially funded by the NSF grant DMS–2324929. We also thank the Simons-Laufer Mathematical Research Institute (SLMath) which hosted our group during the SRiM 2024 program, with additional funding from the NSF grant DMS–1928930.

Furthermore, Kriti Goel was supported by the grant CEX2021-001142-S (Excelencia Severo Ochoa), funded by MICIU/AEI/10.13039/501100011033 (Ministerio de Ciencia, Innovaci\'on y Universidades, Spain). 
Kuei-Nuan Lin was partially supported by the NSF-AWM Travel Grant Program 2023 and the AMS-Simons Research Enhancement Grants for Primarily Undergraduate Institution Faculty 2024. 
Haydee Lindo was supported by NSF grant DMS-2137949.
Whitney Liske's participation was made possible in part by a grant from the Association for Women in Mathematics’ Mathematical Endeavors Revitalization Program.
Maral Mostafazadehfard was partially funded by CAPES–Brasil (Finance Code 001) and the MathAmSud project ALGEO.  

Finally, it is our pleasure to thank Ben Blum-Smith, Matt Mastroeni, Yi-Huang Shen, and Gabriel Sosa for helpful and insightful discussions, and the anonymous referees for their useful comments that improved the exposition of the paper.
\end{acknowledgment*}

\bibliography{ReesBib}

@article {BLHibi,
    AUTHOR = {Brown, J. and Lakshmibai, V.},
     TITLE = {Singular loci of {G}rassmann-{H}ibi toric varieties},
   JOURNAL = {Michigan Math. J.},
  FJOURNAL = {Michigan Mathematical Journal},
    VOLUME = {59},
      YEAR = {2010},
    NUMBER = {2},
     PAGES = {243--267},
      ISSN = {0026-2285,1945-2365},
   MRCLASS = {14M25 (14M15)},
  MRNUMBER = {2677619},
MRREVIEWER = {David\ E.\ Anderson},
       DOI = {10.1307/mmj/1281531454},
       URL = {https://doi.org/10.1307/mmj/1281531454},
}

@book {HS06,
    AUTHOR = {Huneke, Craig and Swanson, Irena},
     TITLE = {Integral closure of ideals, rings, and modules},
    SERIES = {London Mathematical Society Lecture Note Series},
    VOLUME = {336},
 PUBLISHER = {Cambridge University Press, Cambridge},
      YEAR = {2006},
     PAGES = {xiv+431},
      ISBN = {978-0-521-68860-4; 0-521-68860-4},
   MRCLASS = {13B22 (13A18 13A30 13A35 13H15 14A05)},
  MRNUMBER = {2266432},
MRREVIEWER = {Liam O'Carroll},
}

@Article{Boocher,
  author     = {Boocher, Adam},
  title      = {Free resolutions and sparse determinantal ideals,},
journal    = {Math. Res. Lett.},
  year       = {2012},
  volume     = {19},
  pages      = {805--821},
}

@article{BAM,
 AUTHOR = {Bivi\'{a}-{A}usina, Carles and Monta\~{n}o, Jonathan},
 TITLE = {Analytic spread and integral closure of integrally decomposable modules},
 JOURNAL = {Nagoya Math. J},
 FJOURNAL = {Nagoya Mathematical Journal},
 VOLUME = {245},
 YEAR = {2022},
 PAGES = {166--161},
 ISSN = { },
  MRCLASS = {13B22,13H15, 32S05},
  MRNUMBER = {MR4413366},
       DOI = {doi:10.1017/nmj.2020.35},
       URL = {https://doi.org/10.1017/nmj.2020.35},
}

@Article{CDFGLPS,
    AUTHOR = {Celikbas, Ela and Dufresne, Emilie and Fouli, Louiza and Gorla, Elisa and Lin, Kuei-Nuan and Polini, Claudia and Swanson, Irena},
     TITLE = {Rees algebras of sparse determinantal ideals},
   JOURNAL = {Trans. Amer. Math. Soc.},
  FJOURNAL = {Transactions of the American Mathematical Society},
    VOLUME = {377},
      YEAR = {2024},
     PAGES = {2317--2333},
      ISSN = {0002-9947,1088-6850},
   MRCLASS = {99-06},
  MRNUMBER = {4744759},
       DOI = {10.1090/tran/9101},
       URL = {https://doi.org/10.1090/tran/9101},
}

@Book{BHCMbook,
  author    = {Bruns, Winfried and Herzog, J{\"u}rgen},
  publisher = {Cambridge University Press},
  title     = {Cohen-{M}acaulay rings},
  year      = {1998},
  address   = {Cambridge},
  isbn      = {0-521-56674-6/pbk},
  series    = {Cambridge Studies in Advanced Mathematics},
  volume    = {39},
  language  = {English},
  msc2010   = {13C14 13H10 13D03 13-02 13D25},
  pages     = {xiv + 453},
  zbl       = {0909.13005},
}

@article{CVGBDeform,
    AUTHOR = {Conca, Aldo and Varbaro, Matteo},
     TITLE = {Square-free {G}r\"{o}bner degenerations},
   JOURNAL = {Invent. Math.},
  FJOURNAL = {Inventiones Mathematicae},
    VOLUME = {221},
      YEAR = {2020},
     PAGES = {713--730},
      ISSN = {0020-9910},
   MRCLASS = {13P10},
  MRNUMBER = {4132955},
       DOI = {10.1007/s00222-020-00958-7},
       URL = {https://doi.org/10.1007/s00222-020-00958-7},
}

@Article{ConcaLadder,
    AUTHOR = {Conca, Aldo},
     TITLE = {Ladder determinantal rings},
   JOURNAL = {J. Pure Appl. Algebra},
  FJOURNAL = {Journal of Pure and Applied Algebra},
    VOLUME = {98},
      YEAR = {1995},
     PAGES = {119--134},
      ISSN = {},
   MRCLASS = {},
  MRNUMBER = {MR1319965},
       DOI = {10.1016/0022-4049(94)00039-L.},
       URL = {},
}

@article{CosLS,
  title={Invariants of multi--fiber rings of lexsegment ideals},
  author={Costantini, Alessandra and Lin, Kuei-Nuan and Sosa Castillo, Gabriel},
  year={2026},
  note={preprint},
}

@incollection {Giusti-Merle,
    AUTHOR = {Giusti, M. and Merle, M.},
     TITLE = {Singularit\'es isol\'ees et sections planes de vari\'et\'es
              d\'eterminantielles. {II}. {S}ections de vari\'et\'es
              d\'eterminantielles par les plans de coordonn\'ees},
 BOOKTITLE = {Algebraic geometry ({L}a {R}\'abida, 1981)},
    SERIES = {Lecture Notes in Math.},
    VOLUME = {961},
     PAGES = {103--118},
 PUBLISHER = {Springer, Berlin},
      YEAR = {1982},
      ISBN = {3-540-11969-8},
   MRCLASS = {14M12 (14B05 14M05)},
  MRNUMBER = {708329},
MRREVIEWER = {Tadeusz\ J\'ozefiak},
       DOI = {10.1007/BFb0071278},
       URL = {https://doi.org/10.1007/BFb0071278},
}

@article{dipasquale2020koszul,
    title={Koszul multi-{R}ees algebras of principal {$L$}-{B}orel Ideals},
    author={DiPasquale, M. and Jabbar Nezhad, B.},
    
    journal={Journal of Algebra},
  volume={581},
  
  pages={353--385},
  year={2021}

}

@article {BHAinvariant,
    author    = {Bruns, Winfried and Herzog, J{\"u}rgen},
     TITLE = {On the computation of {$a$}-invariants},
   JOURNAL = {Manuscripta Math.},
  FJOURNAL = {Manuscripta Mathematica},
    VOLUME = {77},
      YEAR = {1992},
    NUMBER = {2-3},
     PAGES = {201--213},
      ISSN = {0025-2611,1432-1785},
   MRCLASS = {13H10 (13F50)},
  MRNUMBER = {1188581},
MRREVIEWER = {Gert\ Naud\'e},
       DOI = {10.1007/BF02567054},
       URL = {https://doi.org/10.1007/BF02567054},
}

@Article{CNPY,
  author   = {Corso, Alberto and Nagel, Uwe and Petrovi\'c, Sonja and Yuen, Cornelia},
  title    = {Blow-up algebras, determinantal ideals, and {D}edekind--{M}ertens-like formulas},
  journal  = {Forum Math.},
  year     = {2017},
  volume   = {29},
  pages    = {799--830},
  issn     = {0933-7741},
  fjournal = {Forum Mathematicum},
}

@article{CHV96,
  title={Sagbi bases with applications to blow-up algebras},
  author={Conca, Aldo and Herzog, Jurgen and Valla, Giuseppe},
  journal={Journal fur die Reine und Angewandte Mathematik},
  volume={474},
  pages={113--138},
  year={1996},
  publisher={Berlin, W. de Gruyter.}
}

@book{DEP,
    AUTHOR = {De Concini, Corrado and Eisenbud, David and Procesi, Claudio },
     TITLE = {Hodge algebras. With a {F}rench summary. },
   JOURNAL = {Astérisque},
  FJOURNAL = {},
    VOLUME = {91},
      YEAR = {1982},
     PAGES = {87 pp},
      PUBLISHER = {Société Mathématique de France, Paris},
      ISSN = {},

}

@incollection {Eisenbud,
    AUTHOR = {Eisenbud, David},
     TITLE = {Introduction to algebras with straightening laws},
 BOOKTITLE = {Ring theory and algebra, {III} ({P}roc. {T}hird {C}onf.,
              {U}niv. {O}klahoma, {N}orman, {O}kla., 1979)},
    SERIES = {Lect. Notes Pure Appl. Math.},
    VOLUME = {55},
     PAGES = {243--268},
 PUBLISHER = {Dekker, New York},
      YEAR = {1980},
      ISBN = {0-8247-1158-0},
   MRCLASS = {14M12 (06A10 13E15 13H10)},
  MRNUMBER = {584614},
MRREVIEWER = {\c{S}erban\ B\u{a}rc\u{a}nescu},
}

@Book{HHMon,
  author    = {Herzog, J{\"u}rgen and Hibi, Takayuki},
  publisher = {Springer-Verlag London Ltd.},
  title     = {Monomial ideals},
  year      = {2011},
  address   = {London},
  isbn      = {978-0-85729-105-9},
  series    = {Graduate Texts in Mathematics},
  volume    = {260},
  mrclass   = {13D02 (13D40 13F55 13P10)},
  mrnumber  = {2724673},
  pages     = {xvi+305},
}

@incollection {Terai,
    AUTHOR = {Terai, Naoki},
     TITLE = {Alexander duality theorem and {S}tanley-{R}eisner rings},
      NOTE = {Free resolutions of coordinate rings of projective varieties
              and related topics (Japanese) (Kyoto, 1998)},
   JOURNAL = {S\={u}rikaisekikenky\={u}sho K\B{o}ky\={u}roku},
  FJOURNAL = {S\={u}rikaisekikenky\={u}sho K\B{o}ky\={u}roku},
    NUMBER = {1078},
      YEAR = {1999},
     PAGES = {174--184},
   MRCLASS = {13F55 (13D02)},
  MRNUMBER = {1715588},
MRREVIEWER = {Rafael H. Villarreal},
}

@article {KLS,
  author = {Kara, Selvi and Lin, Kuei-Nuan and Sosa, Gabriel},
     TITLE = {Multi-{R}ees Algebras of Strongly Stable Ideals},
   JOURNAL = {Collect. Math. },
  FJOURNAL = {Collectanea Mathematica},
    VOLUME = {74},
      YEAR = {2023},
     PAGES = {247--278},
    
}

@article{LinPolini,
  title={Rees algebras of truncations of complete intersections},
  author={Lin, K.-N. and  Polini, C.},
  journal={Journal of Algebra},
  volume={410},
  pages={36--52},
  year={2014},
  publisher={Elsevier}
}

@article {LSCMNormal,
    AUTHOR = {Lin, Kuei-Nuan and Shen, Yi-Huang},
     TITLE = {Fiber cones of rational normal scrolls are {C}ohen-{M}acaulay},
   JOURNAL = { J. Algebraic Combin.},
  FJOURNAL = {Journal of Algebraic Combinatorics},
    VOLUME = {56},
      YEAR = {2022},
     PAGES = {547--563},
   
}

@incollection {BCresPowers,
    AUTHOR = {Bruns, Winfried and Conca, Aldo},
     TITLE = {Linear resolutions of powers and products},
 BOOKTITLE = {Singularities and computer algebra},
     PAGES = {47--69},
 PUBLISHER = {Springer, Cham},
      YEAR = {2017},
   MRCLASS = {13D02 (13A30 13C40 13P10)},
  MRNUMBER = {3675721},
MRREVIEWER = {Adela N. Vraciu},
}

@article {LSVer,
    AUTHOR = {Lin, Kuei-Nuan and Shen, Yi-Huang},
     TITLE = {Regularity and multiplicity of {V}eronese type algebras},
   JOURNAL = {Canad. Math. Bull.},
  FJOURNAL = {Canadian Mathematical Bulletin. Bulletin Canadien de
              Math\'ematiques},
    VOLUME = {69},
      YEAR = {2026},
    NUMBER = {1},
     PAGES = {150--178},
      ISSN = {0008-4395,1496-4287},
   MRCLASS = {13F65 (05E40 13F55 13H15 14M25)},
  MRNUMBER = {5021601},
       DOI = {10.4153/S0008439525100970},
       URL = {https://doi.org/10.4153/S0008439525100970},
}

@article {GWGraded,
    AUTHOR = {Goto, Shiro and Watanabe, Keiichi},
     TITLE = {On graded rings. {I}},
   JOURNAL = {J. Math. Soc. Japan},
  FJOURNAL = {Journal of the Mathematical Society of Japan},
    VOLUME = {30},
      YEAR = {1978},
     PAGES = {179--213},
      ISSN = {0025-5645},
   MRCLASS = {13H10 (13D03 14B15)},
  MRNUMBER = {494707},
MRREVIEWER = {Gerald S. Garfinkel},
       DOI = {10.2969/jmsj/03020179},
       URL = {https://doi.org/10.2969/jmsj/03020179},
}

@book {EHgbBook,
    AUTHOR = {Ene, Viviana and Herzog, J\"{u}rgen},
     TITLE = {Gr\"{o}bner bases in commutative algebra},
    SERIES = {Graduate Studies in Mathematics},
    VOLUME = {130},
 PUBLISHER = {American Mathematical Society, Providence, RI},
      YEAR = {2012},
     PAGES = {xii+164},
      ISBN = {978-0-8218-7287-1},
   MRCLASS = {13P10 (13-01)},
  MRNUMBER = {2850142},
MRREVIEWER = {P. Schenzel},
       DOI = {10.1090/gsm/130},
       URL = {https://doi.org/10.1090/gsm/130},
}

@book {Sturmfels,
    AUTHOR = {Sturmfels, Bernd},
     TITLE = {Gr\"{o}bner bases and convex polytopes},
    SERIES = {University Lecture Series},
    VOLUME = {8},
 PUBLISHER = {American Mathematical Society, Providence, RI},
      YEAR = {1996},
     PAGES = {xii+162},
      ISBN = {0-8218-0487-1},
   MRCLASS = {13P10 (14M25 52B20)},
  MRNUMBER = {1363949},
MRREVIEWER = {P. Schenzel},
       DOI = {10.1090/ulect/008},
       URL = {https://doi.org/10.1090/ulect/008},
}

@article{MORALES2018350,
title = {Hook formulas for skew shapes {I}. q-analogues and bijections},
journal = {Journal of Combinatorial Theory, Series A},
volume = {154},
pages = {350-405},
year = {2018},
issn = {0097-3165},
doi = {https://doi.org/10.1016/j.jcta.2017.09.002},
url = {https://www.sciencedirect.com/science/article/pii/S0097316517301267},
author = {Alejandro H. Morales and Igor Pak and Greta Panova},
}

@article{Nagel2005,
 AUTHOR = {Nagel, Uwe},
 TITLE = {Comparing {C}astelnuovo-{M}umford regularity and extended degree: The borderline cases},
 JOURNAL = {Trans. Amer. Math. Soc.},
 FJOURNAL = {Transactions of the American Mathematical Society},
 VOLUME = {357},
 YEAR = {2005},
 PAGES = {3585--3603},
 ISSN = {0002-9947,1088-6850},
  MRCLASS = {13D45 (13A02 13D40 13H10},
  MRNUMBER = {2146640},
       DOI = {10.1090/S0002-9947-04-03595-0},
       URL = {https://doi.org/10.1090/S0002-9947-04-03595-0},
}

@article{Nagel1990,
 AUTHOR = {Nagel, Uwe},
 TITLE = {On {C}astelnuovo's regularity and {H}ilbert functions},
 JOURNAL = {Compositio Math.},
 FJOURNAL = {Compositio Mathematica},
 VOLUME = {76},
 YEAR = {1990},
 PAGES = {265--275},
 ISSN = {0010-437X,1570-5846},
  MRCLASS = {14B15 (13D45 13H15)},
  MRNUMBER = {1078866},
       DOI = {},
       URL = {http://www.numdam.org/item/CM_1990__76_1-2_265_0/},
}

@Book{JonathanThesis,
  author    = {Monta\~no, Jonathan},
  publisher = {ProQuest LLC},
  title     = {Generalized multiplicities, reductions of ideals, and depth of blowup algebras},
  year      = {2015},
  address   = {},
  isbn      = {},
  series    = {Thesis (Ph.D.)--Purdue University},
  volume    = {},
  language  = {},
  msc2010   = {},
  pages     = {95},
  zbl       = {},
}

@article {SUVModule,
    AUTHOR = {Simis, Aron and Ulrich, Bernd and Vasconcelos, Wolmer V.},
     TITLE = {Rees algebras of modules},
   JOURNAL = {Proc. London Math. Soc. (3)},
  FJOURNAL = {Proceedings of the London Mathematical Society. Third Series},
    VOLUME = {87},
      YEAR = {2003},
    NUMBER = {3},
     PAGES = {610--646},
      ISSN = {0024-6115,1460-244X},
   MRCLASS = {13A30 (13B21 13H10)},
  MRNUMBER = {2005877},
MRREVIEWER = {Paulo\ F.\ Machado},
       DOI = {10.1112/S0024611502014144},
       URL = {https://doi.org/10.1112/S0024611502014144},
}

@Book{Vasconcelos2005,
  author    = {Vasconcelos, Wolmer},
  publisher = {Springer-Verlag},
  title     = {Integral closure},
  year      = {2005},
  address   = {Berlin},
  isbn      = {978-3-540-25540-6},
  series    = {Springer Monogr. Math.},
  volume    = {},
  language  = {English},
  msc2010   = {13A30,13-02,13B22,13D40,13H15},
  pages     = {xii+519},
  zbl       = {},
}

@article {LinShenLadder,
title = {Blowup algebras of determinantal modules},
journal = {Journal of Algebra},
volume = {697},
pages = {632-659},
year = {2026},
issn = {0021-8693},
doi = {https://doi.org/10.1016/j.jalgebra.2026.02.033},
url = {https://www.sciencedirect.com/science/article/pii/S0021869326001286},
author = {Kuei-Nuan Lin and Yi-Huang Shen}
}

@article {EHU,
    AUTHOR = {Eisenbud, David and Huneke, Craig and Ulrich, Bernd},
     TITLE = {What is the {R}ees algebra of a module?},
   JOURNAL = {Proc. Amer. Math. Soc.},
  FJOURNAL = {Proceedings of the American Mathematical Society},
    VOLUME = {131},
      YEAR = {2003},
    NUMBER = {3},
     PAGES = {701--708},
      ISSN = {0002-9939,1088-6826},
   MRCLASS = {13A30 (13B21 13C12)},
  MRNUMBER = {1937406},
MRREVIEWER = {Irena\ Swanson},
       DOI = {10.1090/S0002-9939-02-06575-9},
       URL = {https://doi.org/10.1090/S0002-9939-02-06575-9},
}

@incollection {Grassmannians,
    AUTHOR = {Smirnov, Evgeny},
     TITLE = {Grassmannians, flag varieties, and {G}elfand-{Z}etlin
              polytopes},
 BOOKTITLE = {Recent developments in representation theory},
    SERIES = {Contemp. Math.},
    VOLUME = {673},
     PAGES = {179--226},
 PUBLISHER = {Amer. Math. Soc., Providence, RI},
      YEAR = {2016},
      ISBN = {978-1-4704-1955-4},
   MRCLASS = {14M15 (05A15 05E05 52B20)},
  MRNUMBER = {3546711},
MRREVIEWER = {Subramaniam\ Senthamarai\ Kannan},
       DOI = {10.1090/conm/673/13491},
       URL = {https://doi.org/10.1090/conm/673/13491},
}

@book {HerHiOh,
    AUTHOR = {Herzog, J\"urgen and Hibi, Takayuki and Ohsugi, Hidefumi},
     TITLE = {Binomial ideals},
    SERIES = {Graduate Texts in Mathematics},
    VOLUME = {279},
 PUBLISHER = {Springer, Cham},
      YEAR = {2018},
     PAGES = {xix+321},
      ISBN = {978-3-319-95347-2; 978-3-319-95349-6},
   MRCLASS = {13F20 (13-02 13P10 13P25 14M25 52B20)},
  MRNUMBER = {3838370},
MRREVIEWER = {Thomas\ Kahle},
       DOI = {10.1007/978-3-319-95349-6},
       URL = {https://doi.org/10.1007/978-3-319-95349-6},
}

@book {FultonYoung,
    AUTHOR = {Fulton, William},
     TITLE = {Young tableaux},
    SERIES = {London Mathematical Society Student Texts},
    VOLUME = {35},
      NOTE = {With applications to representation theory and geometry},
 PUBLISHER = {Cambridge University Press, Cambridge},
      YEAR = {1997},
     PAGES = {x+260},
      ISBN = {0-521-56144-2; 0-521-56724-6},
   MRCLASS = {05E10 (05E05 05E15 14M15 20G05)},
  MRNUMBER = {1464693},
MRREVIEWER = {Tadeusz\ J\'ozefiak},
}

@incollection {HibiDistLatt,
    AUTHOR = {Hibi, Takayuki},
     TITLE = {Distributive lattices, affine semigroup rings and algebras
              with straightening laws},
 BOOKTITLE = {Commutative algebra and combinatorics ({K}yoto, 1985)},
    SERIES = {Adv. Stud. Pure Math.},
    VOLUME = {11},
     PAGES = {93--109},
 PUBLISHER = {North-Holland, Amsterdam},
      YEAR = {1987},
      ISBN = {0-444-70314-4},
   MRCLASS = {13H10 (06D99)},
  MRNUMBER = {951198},
MRREVIEWER = {Keiichi\ Watanabe},
       DOI = {10.2969/aspm/01110093},
       URL = {https://doi.org/10.2969/aspm/01110093},
}

@article {Ribbe,
    AUTHOR = {Ribbe, J.},
     TITLE = {On the defining equations of multi-graded rings},
   JOURNAL = {Comm. Algebra},
  FJOURNAL = {Communications in Algebra},
    VOLUME = {27},
      YEAR = {1999},
    NUMBER = {3},
     PAGES = {1393--1402},
      ISSN = {0092-7872,1532-4125},
   MRCLASS = {13A30},
  MRNUMBER = {1669107},
MRREVIEWER = {L\^e\ Tu\^an\ Hoa},
       DOI = {10.1080/00927879908826501},
       URL = {https://doi.org/10.1080/00927879908826501},
}
\end{document}